\documentclass[a4paper,10pt]{article}
\pdfoutput=1

\usepackage[utf8]{inputenc}

\usepackage{amsmath,amsfonts,amsthm,amssymb}
\usepackage{graphicx}
\usepackage[colorlinks,citecolor=blue]{hyperref}
\usepackage[nameinlink,noabbrev,capitalize]{cleveref}
\usepackage{tabularx}
\numberwithin{equation}{section}

\newtheorem{theorem}{Theorem}[section]
\newtheorem{lemma}[theorem]{Lemma}
\newtheorem{definition}[theorem]{Definition}

\newtheorem{remark}[theorem]{Remark}
\newtheorem{corollary}[theorem]{Corollary}
\newtheorem{example}[theorem]{Example}
\newtheorem{algo}{Algorithm}
\newtheorem{assumption}{Assumption}

\newcommand\be{\begin{equation}}
\newcommand\ee{\end{equation}}

\newcommand{\dx}{\,\text{\rm{}d}x}
\newcommand{\dt}{\,\text{\rm{}d}t}

\def\R{\mathbb  R}

\def\meas{\operatorname{meas}}
\def\proj{\operatorname{proj}}
\def\gph{\operatorname{gph}}

\newcolumntype{L}{>{$}l<{$\quad}}
\newcolumntype{R}{>{$}r<{$\quad}}
\newcolumntype{C}{>{$}c<{$}}
\newcommand{\mc}[1]{\multicolumn{1}{c}{#1}}

\title{Iterative hard-thresholding applied to optimal control problems with $L^0(\Omega)$ control cost}

\author{Daniel Wachsmuth%
\footnote{Institut f\"ur Mathematik,
Universit\"at W\"urzburg,
97074 W\"urzburg, Germany, {\tt daniel.wachsmuth@mathematik.uni-wuerzburg.de}.
This research was partially supported by the German Research Foundation DFG under project grant Wa 3626/1-1.}}

\begin{document}

\maketitle

\paragraph{Abstract.}
We investigate the hard-thresholding method applied to optimal control problems with $L^0(\Omega)$ control
cost, which penalizes the measure of the support of the control. As the underlying measure space is non-atomic,
arguments of convergence proofs in $l^2$ or $\R^n$ cannot be applied.
Nevertheless, we prove the surprising property that the values of the objective functional are lower semicontinuous
along the iterates. That is, the function value in a weak limit point is less or equal than the lim-inf of the function
values along the iterates. Under a compactness assumption, we can prove that weak limit points are strong limit points,
which enables us to prove certain stationarity conditions for the limit points.
Numerical experiments are carried out, which show the performance of the method.
These indicates that the method is robust with respect to discretization.
In addition, we show that solutions obtained by the thresholding algorithm are superior to solutions of $L^1(\Omega)$-regularized problems.


\section{Introduction}

In this article, we are interested in the development of an algorithm to solve
optimization problems of the type
\be\label{eq001}
 \min f(u) + \frac \alpha 2 \|u\|_{L^2(\Omega)}^2 + \beta \|u\|_0
\ee
subject to
\be\label{eq002}
u\in U_{ad} := \{ v\in L^2(\Omega):  |v(x)| \le b \text{ a.e.\ on } \Omega \}.
\ee
Here, $\Omega\subset \R^n$ is an open domain.
The objective functional contains the so-called $L^0$-norm (which is -- of course -- not a norm) that is defined by
\[
 \|u\|_0 := \operatorname{meas} \{ x: \ u(x)\ne 0\} \quad u:\Omega \to \R \text{ measurable.}
\]
The parameters are assumed to satisfy $\alpha\ge0$, $\beta>0$,
and $b\in [0,+\infty]$.
The function $f$ maps from $L^2(\Omega)$ to $\R$.
Here, we have in mind to choose
\[
 f(u) := j(y),
\]
where $y$ is the solution of a possibly nonlinear partial differential equation
\[
 E(y,u)=0.
\]
Then in problem \eqref{eq001}--\eqref{eq002} we are looking for a control $u$ that minimizes
a certain objective functional, where also the measure of the support of $u$ is penalized.
The parameter $\beta$ weights the measure of the support against the other ingredients of the
functional. For large values of $\beta$, one expects that the support of solutions is small,
and the solution is called sparse control.
One application of such problems are actuator location problems, where one tries to find
optimal actuator locations for the controls that are small.
We refer to the seminal paper \cite{Stadler2009}, which addresses this problem by using $\|u\|_{L^1(\Omega)}$
instead of $\|u\|_0$ in the cost functional.
Recently, optimal control problems involving $L^0$-norms were derived to enforce a particular control structure.
We refer to \cite{ClasonItoKunisch2016} for an application to switching control problems
and to \cite{ClasonKunisch2014} for control problems, where the control is allowed to take values only from a finite set.

The optimization problem \eqref{eq001}--\eqref{eq002} is very challenging to analyze due to
the presence of the term $\|u\|_0$. Since the mapping $u\mapsto \|u\|_0$ is not weakly lower
semicontinuous from $L^p(\Omega)$ to $\R$ for all $p\in[1,\infty)$, it is not possible to
prove existence of solutions.
It is possible to prove existence of solutions if one adds an additional $H^1$-control
cost to the functional. We refer to \cite{ItoKunisch2014} for a discussion of this and other
possibilities to obtain existence results.
If the structure of the function $f$ allows for the technique of needle perturbations, then
a local solution satisfies the Pontryagin maximum principle \cite{ItoKunisch2014}, see also \eqref{eqpmp}. Due to the
inherently combinatorial nature of the $L^0$-term, there might be many local solutions.

A popular approach to circumvent the difficulties associated to $\|u\|_0$
is to use $L^1$-norms instead, i.e., to solve problem of the type
\be\label{eq003}
  \min_{u\in U_{ad}} f(u) + \frac \alpha 2 \|u\|_{L^2(\Omega)}^2 + \gamma \|u\|_{L^1(\Omega)}.
\ee
Here, existence of solutions can be proven by the standard direct method of the calculus of variations
if $\alpha>0$ or $b<+\infty$.
If $f$ is convex and $\alpha>0$, then the functional in \eqref{eq003}  is strongly convex, and the resulting
optimization problem is uniquely solvable. Under mild conditions on $\nabla f$, one can prove
that minimizers of \eqref{eq003} are sparse, \cite{WachsmuthWachsmuth2011}. Moreover, there is $\gamma_0>0$ such that
$u=0$ is the unique solution of  \eqref{eq003} for $\gamma \ge\gamma_0$, see, e.g., \cite{Stadler2009}.
This can be interpreted as exact penalization of the constraint $u=0$ by using the $L^1$-norm.
In addition, a solution $u_\gamma$ of \eqref{eq003} for $\alpha=0$ is under a suitable assumption
a solution of \eqref{eq001}--\eqref{eq002}, see \cref{sec23}.
In case $\alpha>0$ such a result is not available, and as numerical results suggest such a result cannot be expected.

In this article, we propose to use a hard-thresholding algorithm to compute a sequence $(u_k)$ of
feasible points of \eqref{eq001}--\eqref{eq002} for which the sequence of objective function values is decreasing.
This algorithm can be interpreted as proximal gradient method.
The proximal gradient method is a first-order method to
solve
\[
 \min_{x\in H} f(x) + g(x).
\]
with smooth $f$ and non-smooth $g$ on the Hilbert space $H$.
Given an iterate $x_k$, the next iterate $x_{k+1}$ is computed as the global solution of
\[
 \min f(x_k ) + \nabla f(x_k)(x-x_k) + \frac L2 \|x-x_k\|_H^2 + g(x),
\]
where $L>0$ is a parameter. Introducing the proximal map
\[
 \operatorname{prox}_{L^{-1}g}(x) = \arg\min \left( \frac 12 \|x-x_k\|_H^2 + L^{-1}g(x) \right),
\]
the iteration above can be written as
\[
 x_{k+1} =  \operatorname{prox}_{L^{-1}g}\left( x_k - \frac1L \nabla f(x_k)\right).
\]
Here, we can see that $L^{-1}$ acts as step length in a gradient step.
For $g=0$ the gradient descent method is recovered. Choosing $g$ as the indicator function of a convex set,
we obtain the projected gradient method.
If $f$ and $g$ are convex, then the iterates of this method converge weakly in $H$ to a minimizer of $f+g$,
for the precise statement we refer to \cite[Corollary 27.9]{BauschkeCombettes2011}.

We will employ this method for the splitting
\[
  f(u) + \left( \frac \alpha 2 \|u\|_{L^2(\Omega)}^2 + \beta \|u\|_0  + I_{U_{ad}}\right) = : f(u) + g(u),
\]
where $I_{U_{ad}}$ is the indicator function of $U_{ad}$.
Due to the special structure of the function $g$, we can compute its proximal
map pointwise.
For $\alpha=0$ and $b=+\infty$, the proximal map $\operatorname{prox}_{L^{-1}g}$ is given by the so-called hard-thresholding
operator. In the case $\alpha>0$, the resulting proximal gradient step can be rewritten in terms of the
hard-thresholding operator, see \cref{sec32}. The convergence of this method
for sparse approximation problems with $l^2$ coefficients
with linear-quadratic $f$ was proven in \cite{BlumensathDavies2008}.
The interpretation as a proximal gradient method we took from \cite{Lu2014}.
The crucial argument in the convergence proofs is that after a finite number of iterations
the support of the iterates does not change anymore. Then the thresholding method
reduces to a gradient method on a fixed subspace, for which convergence is well-known.
Such an argument is not available in our case of $L^0$ minimization on a non-atomic measure space.
In parts, the proof can be carried over to  setting.
Let $(u_k)$ be a sequence generated by our algorithm. Due to the global minimization involved in the algorithm,
we can prove that the sequence $( f(u_k) + g(u_k))$ is decreasing and converging if $f$ is bounded from below, see \cref{lem28}.
In addition, we can show in \cref{thm313} that for each weakly converging subsequence $u_{n_k}\rightharpoonup u^*$ in $L^2(\Omega)$ it holds
\[
 f(u^*) + g(u^*) \le \liminf_{k\to\infty} ( f(u_{n_k}) + g(u_{n_k})) = \lim_{n\to\infty} ( f(u_n) + g(u_n)).
\]
This is surprising, since the functional $g$ is {\em not} sequentially weakly lower semicontinuous on $L^2(\Omega)$.

In the analysis, we require that the parameter $L$ is larger than the Lipschitz constant of $\nabla f$.
Using a suitable decrease condition, we can formulate a step-size selection strategy. The algorithm with this variable step sizes
enjoys the same convergence properties as the algorithm with fixed step size, see \cref{sec33}.

If $\nabla f$ is completely continuous and $\alpha>0$, then we can show that each weak limit point of $(u_k)$ is a strong limit point
in \cref{thm316}.
This limit point satisfies a certain inclusion, which is weaker than
the maximum principle, see \cref{lem317}.

Another idea to solve \eqref{eq001}--\eqref{eq002} is to consider the partially convexified problem,
that is replace $g$ by its biconjugate $g^{**}$.
Under standard assumptions, the partially convexified problem
\be\label{eq004}
 \min_{u\in U_{ad}} f(u) + g^{**}(u)
\ee
has global solutions.
In addition, many methods are available to solve \eqref{eq004}
Moreover, stationary point of the original problem
are stationary points of \eqref{eq004}.
If the original problem is unsolvable, then it is tempting to solve
the convexified problem instead. However in this unsolvable case,
solutions of the convexified problem are not fixed points of the hard-thresholding
method. We prove this surprising result in \cref{thm325}.

We report on the performance of our algorithm in \cref{sec4}.
As it turns out, our algorithm generates points with lower objective values for the original problem \eqref{eq001}--\eqref{eq002}
than the corresponding $L^1$-control problem \eqref{eq003}.

\section{Preliminary results}

\subsection{Notation and assumptions}

Let $\Omega\subset\R^n$ be an open and bounded set.
Let us mention that we can replace the spaces $L^p(\Omega)$ by
spaces of integrable functions $L^p(\mu)$ on an arbitrary measure space $(\Omega,\mathcal A, \mu)$ with $\mu(\Omega)<+\infty$.

\begin{assumption}\label{ass1}

We rely on the following standing assumptions.
\begin{enumerate}
 \item
 $f:L^2(\Omega) \to \R$ is weakly lower semicontinuous and bounded from below,
 \item
 $f:L^2(\Omega) \to \R$ is Fr\'echet differentiable, 
 \item
 $\nabla f$ is Lipschitz continuous from $L^2(\Omega)$ to $L^2(\Omega)$ with modulus $L_f$:
 \[
  \|\nabla f(u_1)-\nabla f(u_2)\|_{L^2(\Omega)} \le L_f \|u_1-u_2\|_{L^2(\Omega)}
 \]
 for all $u_1,u_2\in L^2(\Omega)$.
 \end{enumerate}
\end{assumption}

Here, we have in mind to choose $f$ to be a functional depending on the solution of a partial differential equation,
in which $u$ acts as a control.

\begin{example}\label{ex1}
The following example is covered by \cref{ass1}.
Let us define
\[
 f(u):= \int_\Omega L(x,y_u(x)) \dx,
\]
where $y_u\in H^1_0(\Omega)$ is defined to be the unique weak solution of the elliptic partial differential
equation
\[
 (Ay)(x) + d(x, y(x)) = u(x) \quad \text{ a.e.\ in }\Omega.
\]
Under standard assumptions on $A$, $L$ and $d$ (uniform ellipticity, bounded coefficients, Caratheodory property, differentiability, monotonicity of $d$ with respect to $y$,
boundedness on bounded sets),
see, e.g. \cite{CasasHerzogWachsmuth2012},
one can prove that $f$ satisfies \cref{ass1}.
The maximum principle holds for such problems as well, \cite{BonnansCasas1995}
\end{example}

\begin{example}\label{ex2}
Let $\Omega\subset\R^d$ be a bounded domain, let $T>0$ be given, set $I:=(0,T)$. Consider the parabolic equation
\[
\partial_t y  + Ay = u  \text{ on } I\times\Omega \times I, \ y(0) = y_0
\]
for an elliptic partial differential operator $A$
with homogeneous Dirichlet boundary conditions. For given $u\in L^2(I\times\Omega)$ the equation
admits a unique weak solution $y_u\in L^2(I,H^1_0(\Omega))$ with $\partial_t y_u \in L^2(I, H^{-1}(\Omega))$.
Under suitable assumptions on the integrands in
\[
 f(u) := \int_{I\times \Omega} L(t,x,y(x,t)) \dt\dx + \int_\Omega l(x,y(T,x)) \dx,
\]
the example fits into the framework of this article.
The maximum principle for control of parabolic equations was investigated in \cite{RaymondZidani1998}.
\end{example}

\begin{example}
Let us consider the parabolic equation as in \cref{ex2}.
Here, we want to study an actuator design problem: Determine a function $u_0\in L^2(\Omega)$
with small support together with an amplitude function $u_1\in L^\infty(I)$ such that the state associated to
the control $u(x,t):= u_0(x)u_1(t)$ minimizes $f(u)$.
That is, here we want to solve the following problem:
\[
 \min f(u_0,\, u_1) + \frac{\alpha_0}2 \|u_0\|_{L^2(\Omega)}^2 +  \frac{\alpha_1}2 \|u_1\|_{L^2(I)}^2  + \beta \|u_0\|_0
\]
subject to additional control constraints on $u_0$ and $u_1$.
This example can serve as an extension of so-called directional sparsity control problems \cite{HerzogStadlerWachsmuth2012}.
\end{example}

As already indicated in the introduction, we use a particular splitting of the objective functional.
Let us define
\be\label{eq103}
 g(u):= \frac \alpha 2 \|u\|_{L^2(\Omega)}^2 + \beta \|u\|_0.
\ee
Let us introduce some notation related to the term $\|u\|_0$.
For $u\in \R$ define
\[
 |u|_0 := \begin{cases} 0 & \text{ if } u=0,\\ 1 & \text{ if } u\ne0.\end{cases}
\]
Now, let $u\in L^2(\Omega)$ be given.
Then we set
\[
 \chi(u)(x):= |u(x)|_0 \quad \text{ f.a.a. } x\in \Omega,
\]
which implies
\[
 \|u\|_0 = \|\chi\|_{L^1(\Omega)}.
\]

\subsection{Existence of solutions and necessary optimality conditions}

First, let us show that $u\mapsto \|u\|_0$ is not weakly lower semicontinuous. To this end,
take $\Omega=(0,1)$ and define $u_k(x):=1+\operatorname{sign}(\sin(n\pi x))$.
Then it holds $\|u_k\|_0=1/2$ for all $k$. It is well known that $u_k \rightharpoonup u^*=1$ in $L^p(0,1)$ for all $p\in [1,\infty)$.
This shows $1/2 = \liminf \|u_k\|_0 < \|u^*\|_0$.
Hence, the direct method of the calculus of variations cannot be used to prove existence of solutions of \eqref{eq001}--\eqref{eq002}.

One possible modification of \eqref{eq001}--\eqref{eq002} is
\[
 \min f(u) + \frac \alpha 2 \|u\|_{H^1(\Omega)}^2 + \beta \|u\|_0
\]
with $\alpha>0$,
which was investigated by \cite{ItoKunisch2014}. Minimizing sequences of this problem
are bounded in $H^1(\Omega)$. Hence by compact embeddings one finds a subsequence that converges
weakly in $H^1(\Omega)$, strongly in $L^2(\Omega)$, and pointwise a.e.\ on $\Omega$.
Since $v \mapsto |v|_0$ is lower-semicontinuous on $\R$, this allows to apply standard arguments
to prove existence.
We will not follow this modification, as the proximal map of $\|u\|_0$ in $H^1(\Omega)$ cannot be computed
explicitly.

Necessary optimality conditions were proven in \cite[Thm. 2.2]{ItoKunisch2014} as well.
Using the method of needle perturbations, they prove that a locally optimal control
$\bar u$ satisfies the Pontryagin maximum principle. That is, $\bar u$ satisfies
\be\label{eqpmp}
\bar u(x) = \arg\min_{|u|\le b} \nabla f(\bar u)(x)\cdot u + \frac\alpha2|u|^2 + \beta|u|_0
%
\ee
for almost all $x\in \Omega$.
Conversely, if $\bar u$ is a feasible control satisfying this maximum principle \eqref{eqpmp},
then $\bar u$ is locally optimal provided that $f$ is convex and the measure of the set
$\{x: \ |\nabla f(\bar u)(x)| = \sqrt{2\alpha\beta}\}$ is zero, \cite[Thm. 2.7]{ItoKunisch2014}.
We refer here also to \cite{ChatterjeeNagaharaQuevedoRao2016}, where results of this type were
proven for an ODE control problem with free end-time.
Let us provide a short proof in a simplified situation. Of course, the maximum principle can be proven
for concrete control problems under much weaker assumptions.

\begin{theorem}[Pontryagin maximum principle]
Let $\bar u\in L^\infty(\Omega)$ be a local solution of the $L^0$-control problem \eqref{eq001}--\eqref{eq002} in $L^p(\Omega)$, $1\le p<\infty$.
Let $f$ satisfy
\[
 f(u) - f(\bar u) = \nabla f(\bar u)(u-\bar u)  + o(\|u-\bar u\|_{L^1(\Omega)})
\]
for all $u\in U_{ad}$.
Then the maximum principle \eqref{eqpmp} is satisfied for almost all $x\in \Omega$.
\end{theorem}
\begin{proof}
Let us choose $x\in \Omega$, $r>0$, and $v\in \R$ with $|v|\le b$.
Set $\chi_r:= \chi_{B_r(x)}$ and $u_r:= (1-\chi_r)\bar u + \chi_r v$.
It follows  $u_r\in U_{ad}$.
By construction, we have
\[\begin{split}
\|u_r-\bar u\|_{L^1(\Omega)} &= \|\chi_r(v-\bar u)\|_{L^1(\Omega)} \\
& \le (|v| + \|\bar u\|_{L^\infty(\Omega)}) \|\chi_r\|_{L^1(\Omega)} \\
& = (|v| + \|\bar u\|_{L^\infty(\Omega)}) |B_r(x)|.
\end{split}
\]
This implies $u_r\to \bar u$ in $L^p(\Omega)$ for all $1\le p<\infty$ for $r\searrow 0$.
Hence it holds for all $r$ sufficiently small
\begin{multline*}
 0 \le f(u_r) + g(u_r) - f(\bar u) - g(\bar u)\\
 = \nabla f(\bar u)(u_r-\bar u) + o(\|u_r-\bar u\|_{L^1(\Omega)} )\\+ \frac\alpha2 \|u_r\|_{L^2(\Omega)}^2 - \frac\alpha2 \|\bar u\|_{L^2(\Omega)}^2
 + \beta\|u_r\|_0 - \beta \|u\|_0.
\end{multline*}
Here, we divide by $|B_r(x)|$ and pass to the limit $r\searrow0$.
Due to the estimate of $\|u-\bar u\|_{L^1(\Omega)}$ above and the Lebesgue differentiation theorem, we obtain
\[
 0\le \nabla f(\bar u)(x) \cdot (v-\bar u(x)) +  \frac\alpha2 |v| - \frac\alpha2 |\bar u(x)|
 + \beta|v|_0 - \beta |\bar u(x)|_0
\]
for almost all $x\in \Omega$, which is the claim.
\end{proof}

\subsection{Using $L^1$-optimization in case $\alpha=0$}
\label{sec23}

Following the seminal work \cite{Donoho2006}, it is commonly accepted that
sparse solutions of optimization problems can be obtained
using $L^1$-norms. This is true for problem \eqref{eq001} as well,
however, only for the case $\alpha=0$.
For positive parameter $\alpha>0$, this is no longer true.
Numerical results suggest, that controls obtained as solution of
problems with $L^1$-norm are clearly inferior to these obtained with
the thresholding algorithm, see \cref{sec44}.

Let us consider the $L^1$-control problem
\be\label{eq121}
 \min f(u) + \gamma \|u\|_{L^1(\Omega)}
\ee
subject to
\be\label{eq122}
u\in U_{ad} := \{ v\in L^2(\Omega):  |v(x)| \le b \text{ a.e.\ on } \Omega \}.
\ee
Here we assume $b>0$ and $\gamma>0$.
Due to $b>0$, the feasible set in \eqref{eq122} is bounded, and hence \eqref{eq121}--\eqref{eq122}
is solvable under \cref{ass1} on $f$.
The key observation of the connection between \eqref{eq121}--\eqref{eq122} and \eqref{eq001}--\eqref{eq002} (for $\alpha=0$)
is the following: For $u \in U_{ad}$ we have $\|u\|_{L^1(\Omega)} \le b \|u\|_0$.
If in addition $u$ satisfies
$u(x) \in \{-b,0,b\}$ almost everywhere, then
\[
b\|u\|_0 = \|u\|_{L^1(\Omega)}.
\]

\begin{theorem}
Suppose $\gamma = b^{-1} \beta$.
Let $u_{\gamma}$ be a global (or local in $L^2(\Omega)$) solution of the $L^1$-control problem \eqref{eq121}--\eqref{eq122}.
Suppose that
\[
 u_{\gamma}(x) \in \{-b,0,b\}
\]
holds for almost all $x\in \Omega$. Then $u_{\gamma}$ is a global (or local) solution of the
original control problem \eqref{eq001}--\eqref{eq002}.
\end{theorem}
\begin{proof}
Suppose $u_{\gamma}$ is a global minimum of \eqref{eq121}--\eqref{eq122}.
Let $u\in U_{ad}$ be a feasible control. This implies $\|u\|_{L^1(\Omega)} \le b \|u\|_0$.
By optimality of $u_{\gamma}$, we get
\[
f(u_{\gamma}) + \gamma\|u_{\gamma}\|_{L^1(\Omega)} \le f(u) + \gamma b \|u\|_0 = f(u) + \beta \|u\|_0.
\]
Due to the assumptions on $\gamma$ and $u_{\gamma}$, we get
$ \gamma\|u_{\gamma}\|_{L^1(\Omega)} = \beta \|u\|_0$, which proves the claim.
The claim for local solutions follows by restricting $u$ to a neighborhood of $u_{\gamma}$.
\end{proof}

\section{The iterative hard thresholding method}

We present and analyze the thresholding method with fixed step-size first.
The method with variable step-size is analyzed in \cref{sec33}.
Let us recall the splitting of the objective functional in \eqref{eq001} in $f(u)+g(u)$,
where $g(u)$ is defined in \eqref{eq103}.

\begingroup
\renewcommand\thealgo{\textbf{IHT}}
\begin{algo}[Iterative hard thresholding algorithm]\label{A1}
Choose $L>0$, $u_0\in U_{ad}$. Set $k=0$.
\begin{enumerate}
 \item Compute $u_{k+1}$ as global solution of
 \[
  \min_{u\in U_{ad}} f(u_k) + \nabla f(u_k)(u-u_k) + \frac L2 \|u-u_k\|_{L^2(\Omega)}^2 + g(u).
 \]
 \item Set $k:=k+1$ and go to step 1.
\end{enumerate}
\end{algo}
\endgroup

Despite the non-convexity of $|\cdot|_0$, the subproblem in \cref{A1}
is uniquely solvable. Its solution is given by the so-called hard-thresholding operator in
the case $b=+\infty$ and $\alpha=0$.
Before analyzing this operator, let us state the following elementary result.

\begin{corollary}
 Let $u^*\in U_{ad}$ satisfy the maximum principle \eqref{eqpmp}.
 Then $u^*$ is a fixed point of \cref{A1} for all $L\ge0$.
\end{corollary}

\subsection{Analysis of the scalar-valued case}

The global minimization in step 1 of \cref{A1} can be carried out pointwise.
Hence, we first analyze the corresponding optimization problem in $\R$.

\begin{definition}
 Let $t>0$ be given.  Define the set-valued mapping $H_t :\R \rightrightarrows \R$ as follows
 \[
  H_t(q) :=  \begin{cases}\{ q\} & \text{ if } |q| > t,\\
              \{0,\ q\}  & \text{ if } |q| = t,\\
              \{0\} & \text{ if } |q| < t.\\
             \end{cases}
 \]
\end{definition}

 The mapping $H_t$ is called the hard-thresholding operator.
It can be characterized as the solution mapping of an optimization problem.

\begin{lemma}
Let $s>0$, $q\in \R$ be given.  Then it holds $u\in H_{\sqrt{2s}}(-q)$
if and only if $u$ is a global solution of
\be\label{eq203}
 \min qu + \frac12 u^2 + s|u|_0.
\ee
\end{lemma}
\begin{proof}
 The minimum of the quadratic function $u\mapsto qu + \frac12 u^2 + s$ is at $-q$ with value $-\frac12 q^2 +s$.
 Thus the global minimum of the auxiliary problem \eqref{eq203} is at
 $-q$ if $q^2\ge 2s$. In the case $q^2\le 2s$ the point $u=0$ is a global solution.
\end{proof}

Next we incorporate inequality constraints into \eqref{eq203},
which gives rise to an extension of the operator $H_t$.

\begin{lemma}\label{lem23}
Let $s\ge0$, $b\in(0,+\infty]$, $q\in \R$ be given.  Then
$u$ is a global solution of
\be \label{eq204}
 \min_{u: \ |u|\le b} qu + \frac12 u^2 + s|u|_0.
\ee
if and only if $u$ satisfies
one of the conditions
\begin{enumerate}
 \item\label{lem23_1} $u=-b$ and $q \ge \max(b,b/2+s/b)$,
 \item\label{lem23_2} $u=b$ and $q \le -\max(b,b/2+s/b)$,
 \item\label{lem23_3} $u=-q$ and $ \sqrt{2s}  \le |q| \le b$,
 \item\label{lem23_4} $u=0$, $b\le\sqrt{2s}$, and $|q| \le b/2+s/b$,
 \item\label{lem23_5} $u=0$, $b\ge\sqrt{2s}$,  and $|q| \le \sqrt{2s}$.
\end{enumerate}
\end{lemma}
\begin{proof}
The point $-q$ is the unconstrained minimum of $u\mapsto qu + \frac12 u^2 + s$.
Let us discuss the case $q>0$ only. Hence the global minimum of \eqref{eq204} is either at $0$, $-q$, or $-b$.
The value of the objective function at $u=-b$ is equal to $-qb + \frac12 b^2 + s$.
If $-q\le-b$ then $-qb + \frac12 b^2 + s$ is non-positive  if and only if
$q \ge b/2+s/b$.

Hence $u=-b$ is the global solution if $q\ge \max(b,b/2+s/b)$.
If $b\le q\le b/2+s/b$ then $u=0$ is optimal. In case $q<b$ the point $H_{\sqrt{2s}}(-q)$ is optimal.
This implies $u=0$ if and only if $0\le q \le \min(\sqrt{2s},b)$ or $b\le q \le b/2+s/b$ holds.
Due to $b>0$, it holds $\sqrt{2s} \le b/2+s/b$.
Since $b<\sqrt{2s}$ is equivalent to $b<b/2+s/b$, we obtain conditions \ref{lem23_4} and \ref{lem23_5}.
\end{proof}

\begin{definition}
 Let $s>0,b>0$ be given.  Let us define the set-valued mapping $H_{s,b} :\R \rightrightarrows \R$ as follows
 \[
  u\in H_{s,b}(q) \ \Leftrightarrow \ (u,-q) \text{ satisfies conditions \ref{lem23_1}--\ref{lem23_5} of \cref{lem23}}.
 \]
\end{definition}

\begin{corollary}\label{cor35}
 The mapping $q\rightrightarrows H_{s,b}(q)$ is monotone with closed graph.
\end{corollary}
\begin{proof}
Let us first argue that  $\sqrt{2s} < b < b/2 + s/b$ is impossible.
Indeed, these inequalities imply $b^2 > 2s$ and $b^2 < b^2/2 + s$, where the latter is equivalent to
$b^2 < 2s$.

Second, let us consider the exceptional case $\sqrt{2s}=b$, which is equivalent to $b=b/2+s/b$. In this case $H_{s,b}$ is monotone.

Third, let us assume $b < b/2 + s/b$ and $\sqrt{2s}\ne b$. Then $b< \sqrt{2s}$ follows.
Hence item \ref{lem23_3} is void, and we obtain monotonicity, see also the left plot of  \cref{fig1}.

Finally, suppose $b > b/2 + s/b$. Then $b/2 + s/b$ disappears in the system of \cref{lem23},
and monotonicity is proven. This case can be seen in the right plot of  \cref{fig1}.
\end{proof}

\begin{figure}[htb]
 \includegraphics[width=0.49\textwidth]{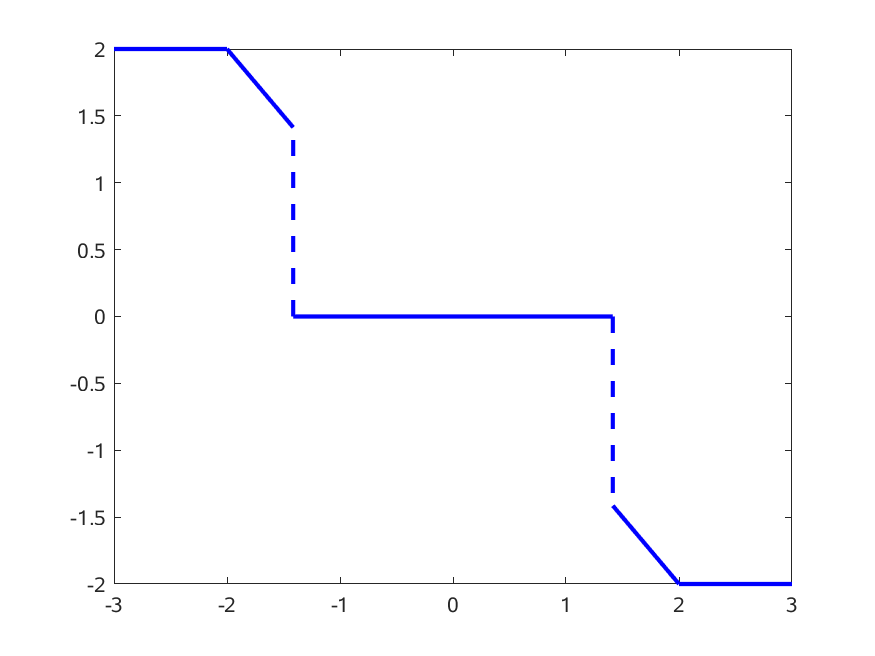}
 \includegraphics[width=0.49\textwidth]{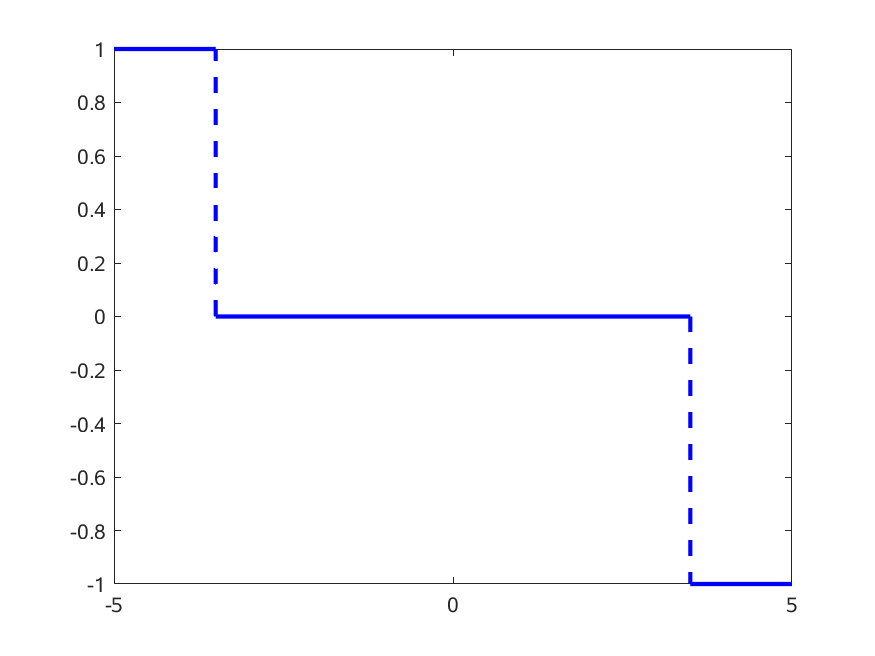}
 \caption{The set-valued map $H_{s,b}$ and its maximal monotone extension $\tilde H_{s,b}$ (dashed lines)
 for parameters $(s,b) = (1,2)$ (left) and $(s,b)=(3,1)$ (right).}
 \label{fig1}
\end{figure}

We see that $q\rightrightarrows H_{s,b}(q)$ is monotone with closed graph but not maximal monotone.
For the unconstrained case $b=+\infty$, we have $H_{s,+\infty}(q) = H_{\sqrt{2s}}(-q)$,
where we used the convention $\max(+\infty, +\infty/2)=+\infty$,
so that items \ref{lem23_1}, \ref{lem23_2}, and \ref{lem23_4} in \cref{lem23} are void.

\begin{definition}
 For $s>0$ and $b\in [0,+\infty]$ we will denote by $\tilde H_{s,b}$ the maximal monotone extension
 of $H_{s,b}$, see also \cref{fig1}.
\end{definition}

Since $H_{s,b}(q)\ne\emptyset$  for all $q\in \R$, it follows that the maximal monotone extension $\tilde H_{s,b}$
is uniquely defined, cf., \cite[Thm.\ 1]{Qi1983}

We close the section with an application to the minimization of
a real-valued prototype of the functional, which is minimized
in each step of the algorithm.

\begin{corollary}\label{cor25}
Let $\beta>0$, $b\in[0,+\infty)$, $g_k\in \R$, $L>0$, $\alpha\ge0$ be given.  Then
$u\in \R$ is a global solution of
\be \label{eq205}
 \min_{u: \ |u|\le b} g_ku + \frac L2 (u-u_k)^2 + \frac\alpha2|u|^2+ \beta|u|_0.
\ee
if and only if $u$ satisfies
 \[
  u \in H_{ \frac\beta{L+\alpha},b}\left(\frac1{L+\alpha}(L u_k-  g_k)\right).
 \]
In addition, all global solutions $u$ of \eqref{eq205}
satisfy
\[
 u=0 \text{ or } |u|\ge \sigma
\]
with $\sigma:= \min(b,\sqrt{\frac{2\beta}{L+\alpha} })$.

Moreover, if $u\ne0$ is a global solution of \eqref{eq205} then it holds
\be\label{eq206}
 (g_k + L (u-u_k) + \alpha u)\cdot (v-u) \ge 0
\ee
for all $|v|\le b$.
\end{corollary}
\begin{proof}
 The minimization problem \eqref{eq205} is equivalent to
 \[
   \min_{u: \ |u|\le b} -\frac1{L+\alpha}(L u_k-  g_k)u + \frac 12 u^2 + \frac\beta{L+\alpha}|u|_0.
 \]
 Then the first claim follows from the construction of $H_{ \frac\beta{L+\alpha},b}$.
 In addition, the second claim follows from \cref{lem23}.
 For a different derivation of $\sigma$ see \cite[Lemma 3.3(i)]{Lu2014}.

 If $u\ne0$ is a global solution of  \eqref{eq205} then it is also a global solution of
 \eqref{eq205} for $\beta=0$. This is a convex problem, and its necessary optimality condition is
 the claimed inequality.
\end{proof}

\subsection{Analysis of the IHT iteration}
\label{sec32}

Let us transfer the results of the previous section to the infinite-dimensional setting.
Recall that $u_{k+1}$ is chosen as a global solution of
\be\label{eq221}
  \min_{u\in U_{ad}} f(u_k) + \nabla f(u_k)(u-u_k) + \frac L2 \|u-u_k\|_{L^2(\Omega)}^2 + g(u).
\ee
Under our standing assumptions, the functional can be written as
an integral function. The integrand can be minimized explicitly.
The following result characterizes all solutions of \eqref{eq221},
and shows that these can be computed.

\begin{lemma}
 Let $u_k \in U_{ad}$ be given. Then the minimization problem
 \be\label{eq222}
  \min_{u\in U_{ad}} f(u_k) + \nabla f(u_k)(u-u_k) + \frac L2 \|u-u_k\|_{L^2(\Omega)}^2 +  g(u).
 \ee
 is solvable. A function $u_{k+1}\in U_{ad}$ is a global solution of this problem if and only if
 \be\label{eq223}
 u_{k+1}(x) \in
 H_{ \frac\beta{L+\alpha},b}\left(\frac1{L+\alpha}\left(L u_k(x)-  \nabla f(u_k)(x)\right)\right)
  \quad \text{ for a.a.\ } x\in \Omega.
\ee
 \end{lemma}
\begin{proof}
 The minimization of \eqref{eq222} is equivalent to the minimization
 \[
  \int_\Omega \nabla f(u_k)(x) u(x) + \frac L2 (u(x)-u_k(x))^2 +\frac\alpha2 |u(x)|^2+ \beta |u(x)|_0\dx.
 \]
Then \cref{cor25} shows that \eqref{eq223} is sufficient.
In addition, with the choice $u_{k+1}(x)=0$ in the case that
$H_{ \frac\beta{L+\alpha},b}\left(\frac1{L+\alpha}\left(L u_k(x)-  \nabla f(u_k)(x)\right)\right)$
is multi-valued, the resulting function $u_{k+1}$ is measurable.
Hence, the problem \eqref{eq222} is solvable.
Let now $u_{k+1}$ be a solution of \eqref{eq222}.
A standard argument shows that \eqref{eq223} is also necessary:
If \eqref{eq223} is violated on a set $A$ of positive measure, then we can modify $u_{k+1}$
on $A$ to get a decrease in the functional \eqref{eq222}, which contradicts the optimality of
$u_{k+1}$.
\end{proof}
This result shows that \cref{A1} is well-defined.
Let us introduce the following notation. We define
 \be\label{eq225}
  I(u):= \{x: \ u(x)\ne 0\}
 \ee
and
 \be\label{eq226}
  \chi_k:= \chi(u_k) = \chi_{I(u_k)}.
 \ee
Then it holds $\|u\|_0=\|\chi_{I(u)}\|_{L^1(\Omega)}$.
We also have the following necessary optimality condition for \eqref{eq222}.

\begin{lemma}\label{lem29}
 Let $u_{k+1}$ solve \eqref{eq222}. Then it holds
 \[
  ( \nabla f(u_k) + L(u_{k+1}-u_k) + \alpha u_{k+1}) (u-u_{k+1})\chi_{k+1} \ge0\quad \forall u\in U_{ad}.
 \]
\end{lemma}
\begin{proof}
 This is a direct consequence of \cref{cor25} and variational inequality \eqref{eq206}.
\end{proof}
We will show that the characteristic functions $\chi_k$ converge in $L^1(\Omega)$.
The key observation is the following result, which is inspired by \cite[Lemma 3.3]{Lu2014}.
We extend it from $\R^n$ to $L^p(\Omega)$.
\begin{lemma}\label{lem27}
 Let $u_k,u_{k+1}\in U_{ad}$, $k\ge1$, be two consecutive iterates of \cref{A1}.
 Then it holds
 \[
  \|u_{k+1}-u_k\|_{L^p(\Omega)}^p \ge \sigma^p\|\chi_k-\chi_{k+1}\|_{L^1(\Omega)}
 \]
 for all $p\in [1,\infty)$, where $\sigma$ is given by \cref{cor25}.
\end{lemma}
\begin{proof}
 On $I(u_k)\setminus I(u_{k+1})$ it holds $u_k(x)\ne0$, $u_{k+1}(x)=0$.
 Hence $|u_{k+1}(x)-u_k(x)|\ge \sigma$ on $I(u_k)\setminus I(u_{k+1})$ by \cref{cor25}.
 This implies
 \[\begin{split}
  \int_\Omega |u_{k+1}-u_k|^p &\dx \ge \sigma^p \meas\Big((I(u_k)\setminus I(u_{k+1})) \cup (I(u_{k+1})\setminus I(u_k))\Big)\\
  &= \sigma^p  \|\chi_k-\chi_{k+1}\|_{L^1(\Omega)},
  \end{split}
 \]
which is the claim.
 \end{proof}

\begin{theorem}\label{lem28}
Suppose $L>L_f$.
 Let $(u_k)$ be a sequence of iterates generated by \cref{A1}.
 Then it holds:
 \begin{enumerate}
  \item The sequences $(u_k)$ and $(\nabla f(u_k))$ are bounded in $L^2(\Omega)$ if $\alpha>0$ or $b<+\infty$.
  \item The sequence $(f(u_k) + g(u_k))$ is monotonically decreasing and converging.
  \item $\|u_{k+1}-u_k\|_{L^2(\Omega)} \to 0$.
  \item $\sum_{k=1}^\infty\|\chi_k-\chi_{k+1}\|_{L^1(\Omega)}< + \infty$.
  \item $\chi_k \to \chi$ in $L^1(\Omega)$ for some characteristic function $\chi$.
 \end{enumerate}

\end{theorem}
\begin{proof}
 We follow the proof of \cite[Theorem 3.4]{Lu2014}.
Due to the Lipschitz continuity of $\nabla f$, we obtain
\[
 f(u_{k+1}) \le f(u_k) + \nabla f(u_k)(u_{k+1}-u_k) + \frac{L_f}2\|u_{k+1}-u_k\|_{L^2(\Omega)}^2.
\]
Since $u_{k+1}$ solves \eqref{eq221}, we have
\begin{multline}\label{eq227}
 f(u_{k+1}) + g(u_{k+1}) \\
 \begin{aligned}
  &\le f(u_k) + \nabla f(u_k)(u_{k+1}-u_k) + \frac{L_f}2\|u_{k+1}-u_k\|_{L^2(\Omega)}^2 + g(u_{k+1})\\
 & = f(u_k) + \nabla f(u_k)(u_{k+1}-u_k) + \frac L2\|u_{k+1}-u_k\|_{L^2(\Omega)}^2 + g(u_{k+1}) \\
 & \qquad - \frac{L-L_f}2\|u_{k+1}-u_k\|_{L^2(\Omega)}^2\\
 & \le f(u_k) +g(u_k) - \frac{L-L_f}2\|u_{k+1}-u_k\|_{L^2(\Omega)}^2 .
 \end{aligned}
\end{multline}
This implies that the sequence $(f(u_k)+ g(u_k) )$ is monotonically decreasing.
Since $f$ and $g$ are bounded from below, it is convergent.
In the case $\alpha>0$, the boundedness of $(u_k)$ follows from the weak coercivity of $g$ in $L^2(\Omega)$.
The Lipschitz continuity of $\nabla f$ then implies the boundedness of $(\nabla f(u_k))$
by $\|\nabla f(u_k)\|_{L^2(\Omega)} \le \|\nabla f(0)\|_{L^2(\Omega)}+L_f\|u_k\|_{L^2(\Omega)}$.

Summing \eqref{eq227} over $k=1\dots n$ yields
\[
 f(u_{n+1}) + g(u_{n+1}) + \frac{L-L_f}2\sum_{k=1}^n \|u_{k+1}-u_k\|_{L^2(\Omega)}^2 \le f(u_0) + g(u_0).
\]
We get by passing to the limit $n\to\infty$
\[
 \lim_{n\to\infty}(f(u_n) + g(u_n) )  + \frac{L-L_f}2\sum_{k=1}^\infty \|u_{k+1}-u_k\|_{L^2(\Omega)}^2
 \le f(u_0) + g(u_0).
\]
Hence, it holds $\|u_{k+1}-u_k\|_{L^2(\Omega)} \to 0$.
Due to \cref{lem27}, it follows
\[
  \frac{L-L_f}2\sigma^2\sum_{k=1}^\infty \|\chi_k-\chi_{k+1}\|_{L^1(\Omega)} \le f(u_0) +g(u_0)
  -  \lim_{n\to\infty}(f(u_n) + g(u_n) ),
\]
which implies that $(\chi_k)$ is a Cauchy sequence in $L^1(\Omega)$.
It follows $\chi_k \to \chi$ in $L^1(\Omega)$, where $\chi$ is the characteristic function of a measurable subset of $\Omega$.
\end{proof}

\begin{remark}
 The property $\chi_{k+1}-\chi_k\to0$ was used in \cite{Lu2014} to show convergence of $(u_k)$.
 In the case of $\R^n$, $\chi_k\in \{0,1\}^n$, it follows $I_k=I_{k+1}$ for all $k\ge k_0$.
 This means, the non-zero entries are identified after a finite number of iterations.
 Then the algorithm reduces to a gradient projection algorithm, whose convergence properties are well-known.
 We cannot apply this argumentation, as the underlying measure of our space $L^p(\Omega)$ is the
 Lebesgue measure, which is non-atomic.
\end{remark}

Although the mapping $u\mapsto \|u\|_0$ is not sequentially weakly lower semicontinuous from $L^2(\Omega)$
to $\R$, we can prove  that the objective functional is weakly lower semicontinuous along the iterates $(u_k)$.
That is, for each weak limit point $u^*$ the value of the objective is lower than the limit $f(u_k)+ \|u_k\|_0$.

\begin{theorem}\label{thm313}
Let $u^*\in U_{ad}$ be a weak sequential limit point of the iterates $(u_k)$ of \cref{A1}
in $L^2(\Omega)$.
Then it holds
\[
 f(u^*)+g(u^*) \le \lim_{k\to\infty} (f(u_k) +g(u_k))
\]
and
\[
(1-\chi)u^*=0
\]
with $\chi$ as in \cref{lem28}.
\end{theorem}
\begin{proof}
 Let $(u_{k_n})$ be a subsequence such that $u_{k_n}\rightharpoonup u^*$ in $U_{ad}$.
 Due to the construction of $\chi_k$, we have $(1-\chi_k)u_k=0$ almost everywhere in $\Omega$.
 Let $\phi\in C_c^\infty(\Omega)$ be given. Then it holds
 \[
  \int_\Omega \phi (1-\chi_k)u_k\dx=0.
 \]
 Due to \cref{lem28} we have $\chi_k\to \chi$ in $L^1(\Omega)$. Since $(\chi_k)$ are characteristic functions,
 it follows $\chi_k\to \chi$  in $L^2(\Omega)$. This allows to pass to the limit in the integral to obtain
 \[
  \int_\Omega \phi (1-\chi)u^*\dx=0.
 \]
 As $\phi \in C_c^\infty(\Omega)$ is arbitrary, it follows $(1-\chi)u^*=0$  almost everywhere in $\Omega$.
 This implies $\chi(u^*) \le \chi$ and $\|u^*\|_0 = \|\chi(u^*)\|_{L^1(\Omega)} \le \|\chi\|_{L^1(\Omega)}$.
In addition, we get
\begin{multline*}
  \lim_{k\to\infty} (f(u_k) +\frac\alpha2\|u_k\|_{L^2(\Omega)}^2+\beta \|u_k\|_0)\\
\begin{aligned}
  &= \liminf\limits_{k_n\to\infty} \left( f(u_{k_n})+\frac\alpha2\|u_k\|_{L^2(\Omega)}^2 \right) + \beta (\lim_{k_n\to\infty}\|\chi_{k_n}\|_{L^1(\Omega)})\\
&  \ge f(u^*) +\frac\alpha2\|u^*\|_{L^2(\Omega)}^2+ \beta\|\chi\|_{L^1(\Omega)}\\
 & \ge f(u^*) +\frac\alpha2\|u^*\|_{L^2(\Omega)}^2+\beta\|u^*\|_0,
\end{aligned}
\end{multline*}
which proves the claim.
\end{proof}

In the next result of this section, we prove that we can pass to the limit in the necessary optimality
condition of \cref{lem29}.
To this end, we need to assume additional properties of $f$.

\begin{lemma}\label{lem314}
Let $p>2$.
Let us assume complete continuity of $\nabla f$ from $L^2(\Omega)$ to $L^p(\Omega)$.
In addition, assume that $(u_k)$ is bounded in $L^p(\Omega)$.
Let $u^*\in U_{ad}$ be a weak sequential limit point in $L^2(\Omega)$ of the iterates $(u_k)$ of \cref{A1}.
Then it holds
\[
 (\nabla f(u^*) + \alpha u^*) (u-u^*)\chi\ge0 \quad \forall u\in U_{ad} \cap L^p(\Omega).
\]
where $\chi$ is as in \cref{lem28}.
\end{lemma}
\begin{proof}
Let $u\in U_{ad} \cap L^p(\Omega)$ and
 $\phi\in C_c^\infty(\Omega)$ with $\phi\ge0$ be given. Then it holds
by \cref{lem29}
\be\label{eq231}
  \int_\Omega (\nabla f(u_k) + L(u_{k+1}-u_k) + \alpha u_{k+1}) (u-u_{k+1}) \chi_{k+1} \phi  \dx\ge 0.
\ee
This is equivalent to
\[
\begin{aligned}
0 &\le \int_\Omega (\nabla f(u_k) + (L+\alpha)(u_{k+1}-u_k) + \alpha u_k) (u-u_k + u_k -u_{k+1}) \chi_{k+1} \phi  \dx\\
&= \int_\Omega (\nabla f(u_k) + \alpha u_k) (u-u_k) \chi_{k+1} \phi  \dx\\
&\qquad
+ \int_\Omega (L+\alpha)(u_{k+1}-u_k)  (u-u_k) \chi_{k+1} \phi  \dx\\
&\qquad
+  \int_\Omega (\nabla f(u_k) + (L+\alpha)(u_{k+1}-u_k) + \alpha u_k) ( u_k -u_{k+1}) \chi_{k+1} \phi  \dx
.
\end{aligned}
\]
Due to \cref{lem28}, we have $u_{k+1}-u_k\to 0$ in $L^2(\Omega)$.
Since $(u_k)$ and $(\nabla f(u_k))$ are bounded in $L^2(\Omega)$ and $(\chi_k)$ is bounded in $L^\infty(\Omega)$,
the second and third integral in this expression tend to zero for $k\to \infty$.
It remains to study the convergence of
\[
\int_\Omega (\nabla f(u_k) + \alpha u_k) (u-u_k) \chi_{k+1} \phi  \dx.
\]
Let $(u_{k_n})$ be a subsequence such that $u_{k_n}\rightharpoonup u^*$ in $U_{ad}$.
Due to the assumptions on $f$, we obtain $\nabla f(u_{k_n})\to \nabla f(u^*)$ in $L^p(\Omega)$ for some $p>2$.
As argued in the proof of \cref{thm313}, it holds $\chi_k\to \chi$ in $L^q(\Omega)$ for $1/p+1/2+1/q=0$.
Then we can pass to the limit along the subsequence to obtain
\[
\int_\Omega \nabla f(u_{k_n})  (u-u_{k_n}) \chi_{k_n+1} \phi  \dx \to
\int_\Omega \nabla f(u^*)  (u-u^*) \chi \phi  \dx.
\]
The last step concerns the limit process of
\[
 \int_\Omega (-u_k^2)  \chi_{k+1} \phi  =  \int_\Omega (-u_k^2)(\chi_{k+1}-\chi_k) \phi + \int_\Omega  (-u_k^2)\phi,
\]
where we have used $\chi_k u_k^2 = u_k^2$.
By assumption, the sequence $(u_k)$ is bounded in $L^p(\Omega)$ for some $p>2$.
The first integral tends to zero for $k\to\infty$ due to $\chi_k\to \chi$ in $L^q(\Omega)$
with $2/p + 1/q=1$. Passing to the lim-sup in the second integral yields
\[
 \limsup_{n\to\infty} \int_\Omega  (-u_{k_n}^2)\phi \le \int_\Omega (-u^*)^2 \phi = \int_\Omega (-u^*)^2\chi \phi,
\]
where the last equality follows from $(1-\chi)u^*=0$ by \cref{thm313}. And the claim is proven.
\end{proof}

\begin{remark}
If $f$ is induced by an optimal control problem, then $\nabla f$ consists of the
superposition of solution operators of partial differential equations, which are smoothing
for elliptic and parabolic equations. Hence, the assumption on complete continuity of $\nabla f$ is
not a serious restriction.
In particular, the function $f$ as defined in \cref{ex1} satisfies the assumptions of the previous lemma,
due to the compact embedding of $H^1(\Omega)$ into $L^p(\Omega)$ for some $p>2$,
where $p$ depends on the spatial dimension.
\end{remark}

In the case $\alpha>0$, we can obtain strong convergence of subsequence of $(u_k)$.

\begin{theorem}\label{thm316}
Suppose $\alpha>0$.
Let us assume complete continuity of $\nabla f$ from $L^2(\Omega)$ to $L^2(\Omega)$.
Let $u^*\in U_{ad}$ be a weak sequential limit point in $L^2(\Omega)$ of the iterates $(u_k)$ of \cref{A1}.
Then $u^*$ is a strong sequential limit point of $(u_k)$ in $L^q(\Omega)$ for all $q<2$.
Moreover, $u^*$ is a fixed point of the hard thresholding iteration, i.e., it satisfies
\be\label{eq235}
 u^*(x) \in
 H_{ \frac\beta{L+\alpha},b}\left(\frac1{L+\alpha}\left(L u^*(x)-  \nabla f(u^*)(x)\right)\right)
  \quad \text{ for a.a.\ } x\in \Omega.
\ee
\end{theorem}
\begin{proof}
 The necessary optimality condition of \cref{lem29} implies
 \[
 u_{k+1}  = \chi_{k+1}u_{k+1}  =\chi_{k+1} \proj_{U_{ad}}\left( -\frac1\alpha (\nabla f(u_k) + L(u_{k+1}-u_k)\right).
 \]
Due to \cref{lem28}, we have $u_{k+1}-u_k\to 0$ in $L^2(\Omega)$.
Let $u_{k_n}\rightharpoonup u^*$. This implies $u_{k_n+1}\rightharpoonup u^*$, and by the assumptions on $\nabla f$,
$\nabla f(u_{k_n}) \to \nabla f(u^*)$ in $L^2(\Omega)$.
In addition, $\chi_{k+1} \to \chi$ in $L^p(\Omega)$ for all $p<\infty$.
This implies the strong convergence of $\nabla f(u_{k_n}) + L(u_{{k_n}+1}-u_{k_n}) \to \nabla f(u^*)$ in $L^2(\Omega)$,
which in turn gives
\[
 \chi_{k+1} \proj_{U_{ad}}\left( -\frac1\alpha ( \nabla f(u_{k_n}) + L(u_{{k_n}+1}-u_{k_n}) \right)
 \to \chi \proj_{U_{ad}}\left( -\frac1\alpha \nabla f(u^*)\right)
\]
in $L^q(\Omega)$ for all $q<2$. Hence it follows $u_{k_n}\to u^*$ in  $L^q(\Omega)$ for all $q<2$.
In addition, $u^*$ satisfies
\[
  u^* = \chi \proj_{U_{ad}}\left( -\frac1\alpha \nabla f(u^*))\right),
\]
which is equivalent to the result of \cref{lem314}. Passing to  pointwise a.e.\@ converging subsequences
yields the claimed fixed point property as a consequence of \eqref{eq223}
and the closedness of the graph of $H$, cf., \cref{cor35}.
\end{proof}

Unfortunately, the fixed point equation \eqref{eq235} for $u^*$ depends on $L$.
In addition, the fixed point equation does not imply the maximum principle \eqref{eqpmp}
but is strictly weaker for $L>0$.
In this sense, the fixed point equation can be interpreted as an first-order optimality condition,
such an optimality condition was called $L$-stationarity in \cite{BeckEldar2013}.

\begin{lemma}\label{lem317}
Suppose $b\in (0,+\infty]$, $L\ge0$, $\alpha\ge0$, $L+\alpha>0$.
Let  $u^*\in U_{ad}$ satisfy \eqref{eq235}. Then it holds
 \be\label{eq370}
  u^*(x) \in H^{\textrm{FP}}_{\alpha,\beta,L,b}\left( \nabla f(u^*)(x) \right).
 \ee
where $H^{\textrm{FP}}$ is given by the following conditions with $s:=\frac\beta{L+\alpha}$:
\begin{enumerate}
 \item If $\sqrt{2s}\le b$ then
 $u \in H^{\textrm{FP}}_{\alpha,\beta,L,b}( g)$ if one of the following conditions is satisfied
	\begin{enumerate}
	 \item $u=-b$ and $g\ge \alpha b$,
	 \item $u=b$ and  $g\le -\alpha b$,
	 \item $u=0$ and $|g|\le (L+\alpha)\sqrt{2s}$,
	 \item $\alpha>0$, $\alpha u = -g$, $|g|\in \alpha [\sqrt{2s},\,b]$.
	 \item $\alpha=0$, $g=0$, $u\in [\sqrt{2s},b]$.
	\end{enumerate}

 \item If $\sqrt{2s} > b$ then
 $u \in H^{\textrm{FP}}_{\alpha,\beta,L,b}( g)$ if one of the following conditions is satisfied
	\begin{enumerate}
	 \item $u=-b$ and $g\ge (L+\alpha)(b/2 + s/b)-Lb$,
	 \item $u=b$ and  $g\le -( (L+\alpha)(b/2 + s/b)-Lb)$,
	 \item $u=0$ and $|g|\le (L+\alpha)(b/2 + s/b)$,
	 \item $\alpha>0$, $\alpha u = -g$, $|g|\in \alpha [\sqrt{2s},\,b]$.
	 \item $\alpha=0$, $g=0$, $u\in [\sqrt{2s},b]$.
	 \end{enumerate}
\end{enumerate}
\end{lemma}
\begin{proof}
It is enough to study the fixed points of the real-valued operator $H$. To this end, let $u$ solve
\[
  u \in H_{ \frac\beta{L+\alpha},b}\left(\frac1{L+\alpha}(L u -g)\right)
\]
for given $g\in \R$, where we later will replace $g$ by $\nabla f(u^*)(x)$.
We will apply the results of \cref{lem23}, where we have to set $s:=\frac\beta{L+\alpha}$
and $q=-\frac1{L+\alpha}(L u-g)$.

{\it Case 1 ($|u|=b$):}
Suppose $u=-b$. Then by \cref{lem23} (1) this is equivalent to
$-\frac1{L+\alpha}(L (-b)-  g) \ge \max(b, b/2 + s/b)$,
which in turn is equivalent to
\[
g \ge (L+\alpha) \max(b, b/2 + s/b)-Lb .
\]
{\it Case 2 ($u=0$):}
By conditions (4)(5) of \cref{lem23}, this is equivalent to
\[
\begin{gathered}
  b\le\sqrt{2s} \ \Rightarrow \ |g| \le (L+\alpha) (b/2 + s/b),\\
  b \ge \sqrt{2s} \ \Rightarrow \ |g| \le (L+\alpha) \sqrt{2s}.
  \end{gathered}
\]
{\it Case 3 ($0<|u|<b$):}
In this case, \cref{lem23} (3) yields $u=-q = \frac1{L+\alpha}(L u-g)$,
or, equivalently $\alpha u = -g$. In addition, $|q| \in [\sqrt{2s},b]$ follows.
In case $\alpha=0$, we get $g=0$, which requires $|u|\in [\sqrt{2s},b]$.
If $\alpha>0$, then $u=\alpha^{-1}g$ follows, in addition
$\alpha^{-1}|g|\in [\sqrt{2s},b]$ follows.
\end{proof}

\begin{figure}[htb]
 \includegraphics[width=0.49\textwidth]{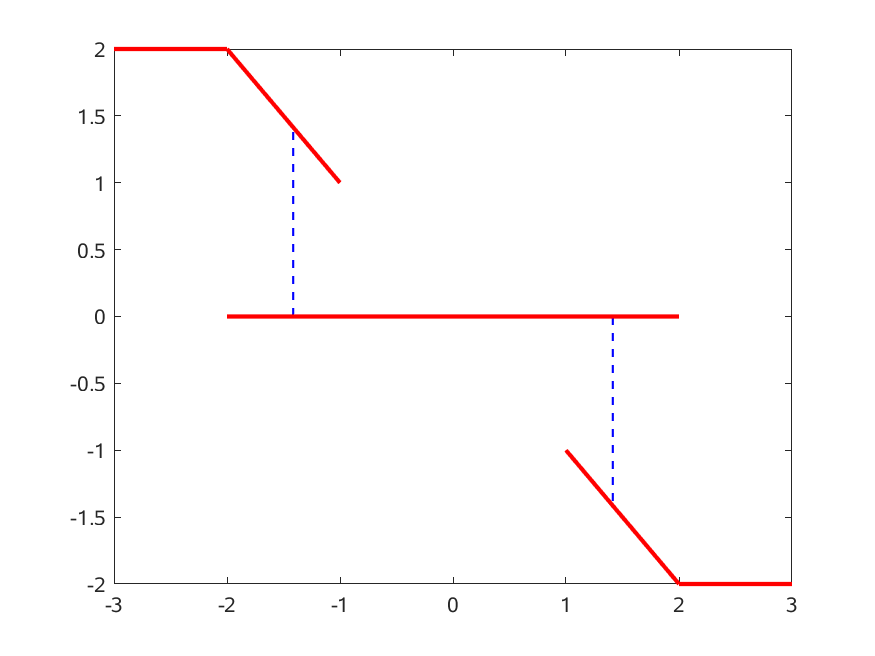}
 \includegraphics[width=0.49\textwidth]{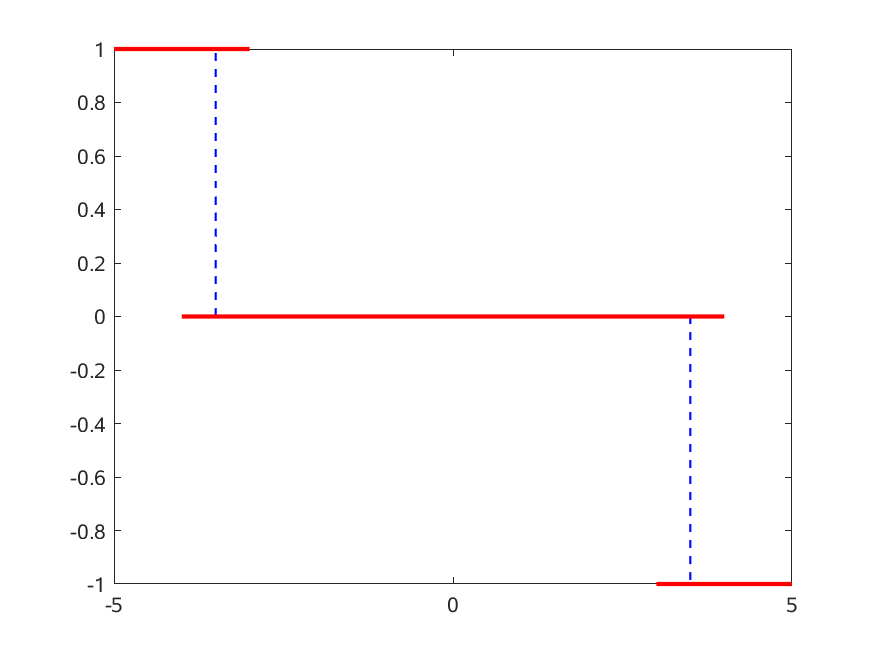}
 \caption{The set-valued map $H^{\textrm{FP}}_{\alpha,\beta,L,b}$ and the maximal monotone map $\tilde H_{\frac\beta{L+\alpha},b}$ (dashed lines)
 for parameters $(\alpha,\beta,L,b) = (1,1,1,2)$ (left) and $(\alpha,\beta,L,b)=(1,3,1,1)$ (right).}
 \label{fig2}
\end{figure}

From the definition, it follows that if $u^*$ satisfies the inclusion \eqref{eq370} with $L=0$ then it satisfies
the maximum principle \eqref{eqpmp}.
In addition, the graph of $H^{\textrm{FP}}_{\alpha,\beta,L,b}$ depends monotonically on $L$.

\begin{lemma}
Suppose $b\in (0,+\infty]$, $L'>L\ge0$, $\alpha\ge0$, $L+\alpha>0$.

Let $u,g\in \R$ such that $u\in H^{\textrm{FP}}_{\alpha,\beta,L,b}(g)$. Then $u\in H^{\textrm{FP}}_{\alpha,\beta,L',b}(g)$.
\end{lemma}
\begin{proof}
 Define $s(L):= \frac\beta{L+\alpha}$. Then $L\mapsto s(L)$ is monotonically decreasing, $L\mapsto (L+\alpha)\sqrt{2s}$
 is monotonically increasing.
 In addition, it holds
 \[
  (L+\alpha)\left(\frac b2 + \frac{s(L)}b\right)-Lb = -\frac b2 L  + \alpha \frac b2 + \frac\beta b,
 \]
hence the mapping $L\mapsto (L+\alpha)(\frac b2 + \frac{s(L)}b)-Lb$ is monotonically decreasing.
This shows that all conditions in \cref{lem317} that are lower bounds on $g$ are monotonically decreasing with $L$,
while all upper bounds on $g$ are monotonically increasing.
And the claimed monotonicity is proven.
\end{proof}

This result suggests that choosing $L$ close to zero is favorable,
which contradicts with the assumption $L>L_f$ in this section.
To overcome this limitation, we will investigate the algorithm with variable step-sizes.

\subsection{IHT with variable step-size}\label{sec33}

Let us introduce the algorithm with variable step-size.
Here, we will replace the assumption $L>L_f$ by a suitable decrease condition, which in the analysis
acts as an replacement of \eqref{eq227} in \cref{lem28}.
In addition, the implementation of the algorithm does not need the knowledge about the Lipschitz constant of $\nabla f$.

\begingroup
\renewcommand\thealgo{\textbf{IHT-LS}}
\begin{algo}[IHT algorithm with variable step-size]\label{A2}
Choose $\eta>0$, $u_0\in U_{ad}$. Set $k=0$.
\begin{enumerate}
 \item Determine $L_k\ge0$ such that the global solution $u_{k+1}$ of
 \[
  \min_{u\in U_{ad}} f(u_k) + \nabla f(u_k)(u-u_k) + \frac {L_k}2 \|u-u_k\|_{L^2(\Omega)}^2 + g(u)
 \]
 satisfies
 \be\label{eq360}
  \eta \|u_{k+1}-u_k\|_{L^2(\Omega)}^2 \le (f(u_k)-g(u_k)) - (f(u_{k+1})+g(u_{k+1})).
 \ee
 \item Set $k:=k+1$ and go to step 1.
\end{enumerate}
\end{algo}
\endgroup

Due to the results of the previous section, the choice
 $L_k \ge L_f + \eta$ satisfies the decrease condition of \cref{A2}.
In numerical computations, we used the following building blocks for a line-search strategy.
Recall that $L_k^{-1}$ corresponds to a step-size.

\begin{enumerate}
 \item Try $L_k=0$ first.
 \item Starting with a given initial guess $\hat L^0$, do an Armijo-like back-tracking: Test values
 $L_k = \hat L^0 \theta^i$ with $\theta>1$.

 \item If initial guess $\hat L^0$ is accepted, then do widening of step-sizes: Test values $L_k = \hat L^0 \theta^{-i}$ with $\theta>1$.
\end{enumerate}

Thanks to the decrease condition \eqref{eq360},
the convergence theory
of the previous section carries over to \cref{A2}.
Note in addition, that no conditions are imposed on the sequence $(L_k)$.

\begin{theorem}
 Let $(u_k)$ be a sequence of iterates generated by \cref{A2}.
 Then it holds:
 \begin{enumerate}
  \item The sequences $(u_k)$ and $(\nabla f(u_k))$ are bounded in $L^2(\Omega)$ if $\alpha>0$ or $b>0$.
  \item The sequence $(f(u_k) + g(u_k))$ is monotonically decreasing and converging.
  \item $\|u_{k+1}-u_k\|_{L^2(\Omega)} \to 0$.
  \item $\sum_{k=1}^\infty\|\chi_k-\chi_{k+1}\|_{L^1(\Omega)}< + \infty$.
  \item $\chi_k \to \chi$ in $L^1(\Omega)$ for some characteristic function $\chi$.
 \end{enumerate}
Let $u^*\in U_{ad}$ be a weak sequential limit point of $(u_k)$ in $L^2(\Omega)$.
Then it holds
\[
 f(u^*)+g(u^*) \le \lim_{k\to\infty} (f(u_k) +g(u_k))
\]
and
\[
(1-\chi)u^*=0.
\]
\end{theorem}
\begin{proof}
The proof is exactly as the proofs of \cref{lem28} and \cref{thm313} with the exception
that \eqref{eq227} has to be replaced by the decrease condition \eqref{eq360}.
\end{proof}

In order to transfer the strong convergence result of \cref{thm316}, we have to study the
upper semi-continuity of
the mapping $s\rightrightarrows \gph H_{s,b}$.

\begin{lemma}\label{lem320}
 The mapping $s\rightrightarrows \gph H_{s,b}$ is upper semi-continuous.
 That is, for sequences $(u_n), (q_n), (s_n)$ of real numbers with $u_n\to u$, $q_n\to q$, $s_n\to s$,
 $s_n\ge0$, $u_n \in H_{s_n,b}(q_n)$  it follows $u\in H_{s,b}(q)$.
\end{lemma}
\begin{proof}
Let us define the set
\[
 H := \{ (u,q,s): \ s\ge0, \ u\in H_{s,b}(q)\}.
\]
Then the claim is equivalent to the closedness of $H$.
This closedness is a direct consequence of the characterization of $H_{s,b}$ in \cref{lem23},
since all the conditions given there are continuous with respect to $(u,q,s)$.
\end{proof}

\begin{theorem}
Suppose $\alpha>0$.
Let us assume complete continuity of $\nabla f$ from $L^2(\Omega)$ to $L^2(\Omega)$.
Let $u^*\in U_{ad}$ be a weak sequential limit point in $L^2(\Omega)$ of the iterates $(u_k)$ of \cref{A2},
that is $u_{k_n} \rightharpoonup u^*$ in $L^2(\Omega)$ for a subsequence.
Assume that the corresponding sequence $(L_{k_n})$ converges to some $L\ge0$.

Then $u^*$ is a strong sequential limit point of $(u_k)$ in $L^q(\Omega)$ for all $q<2$.
Moreover, $u^*$ is a fixed point of the hard thresholding map, i.e., it satisfies
\[
 u^*(x) \in
 H_{ \frac\beta{L+\alpha},b}\left(\frac1{L+\alpha}\left(L u^*(x)-  \nabla f(u^*)(x)\right)\right)
  \quad \text{ for a.a.\ } x\in \Omega.
\]
\end{theorem}
\begin{proof}
The strong convergence follows as in the proof of \cref{thm316}.
The fixed-point property uses the upper-semicontinuity provided by \cref{lem320}.
\end{proof}

\subsection{IHT in the unsolvable case}
\label{sec34}

In this section, we investigate the case that the original problem \eqref{eq001}--\eqref{eq002} is unsolvable.
Here, it is natural to replace $g(u) = \frac\alpha2\|u\|_{L^2(\Omega)} + \beta \|u\|_0$
by its biconjugate (or convexification) given by
\[
 g^{**}(u):= \int_\Omega g_\alpha(u(x))\dx
\]
with integrand
\[
g_\alpha(u) =
\begin{cases}
 \beta + \frac\alpha2 |u|^2 & \text{ if } |u| \ge \sqrt{\frac{2\beta}{\alpha}},\\
 \sqrt{2\alpha\beta}|u| & \text{ if } |u| \le \sqrt{\frac{2\beta}{\alpha}}.
\end{cases}
\]
The resulting functional $g^{**}$ is convex but not strictly convex and continuous from $L^2(\Omega)$ to $\R$.
In addition, we have $g^{**}(u) \ge \frac\alpha2\|u\|_{L^2(\Omega)}^2$.
Hence under our assumptions on $f$, the (partially) convexified problem
\be\label{eq380}
 \min_{u\in U_{ad}} f(u) + g^{**}(u)
\ee
is solvable.
In addition, every stationary point of the original problem \eqref{eq001}--\eqref{eq002}
satisfying the maximum principle \eqref{eqpmp} is a stationary point of \eqref{eq380}.

\begin{lemma}
 Let $\bar u\in U_{ad}$ satisfy \eqref{eqpmp}. Then it holds
 \be\label{eq381}
  0 \in \nabla f(\bar u) + \partial g^{**}(\bar u).
 \ee
\end{lemma}
\begin{proof}
 Since global minimizers of a function are global minimizers of its biconjugate, the maximum principle \eqref{eqpmp}
 implies
\be\label{eq382}
 \bar u(x) = \arg\min_{|u|\le b} \nabla f(\bar u)(x)\cdot u + g_\alpha(u).
\ee
Since $g_\alpha$ is convex, this is equivalent to
\[
 0 \in \nabla f(\bar u)(x) + \partial g_\alpha(\bar u(x)),
\]
which is the claim.
\end{proof}

This result implies that if the original problem is unsolvable, every minimizer of the convexified problem
satisfies \eqref{eq381} and \eqref{eq382} but not the maximum principle \eqref{eqpmp}.
With \cref{cor25} this implies that
\[
 \bar u(x) \not\in H_{\frac\beta\alpha,b} ( - \nabla f(\bar u))(x) )
\]
holds on a set of positive measure. We will now show that such a control $\bar u$ cannot be a fixed point of \cref{A1}.
We first discuss the scalar situation.

\begin{lemma}\label{lem324}
Let $g\in\R$ be given. Let $\bar u\in \R$ be such that
 \[
   \bar u = \arg\min_{|u|\le b} g \cdot u + g_\alpha(u).
 \]
 Let $\bar u$ be not a solution of
\[
 \min_{|u|\le b} g\cdot u + \frac\alpha2 |u|^2 + \beta|u|_0.
\]
Then $\bar u$ is not a global minimum of
\be\label{eq383}
  \min_{|u|\le b} g\cdot (u-\bar u) + \frac L2(u-\bar u)^2 + \frac\alpha2 |u|^2 + \beta|u|_0.
\ee
for all $L>0$.
\end{lemma}
\begin{proof}
By assumption, $\bar u$ is not a global minimum of the function $u\mapsto  g\cdot u + \frac\alpha2 |u|^2 + \beta|u|_0$
but a global minimum  of its convexification $u\mapsto  g \cdot u + g_\alpha(u)$.
Consequently, it follows $g_\alpha(\bar u) < \frac\alpha2 |\bar u|^2 + \beta|\bar u|_0$.
This implies $0<|\bar u|< \sqrt{\frac{2\beta}\alpha}$. The optimality condition
$0\in g + \partial g_\alpha(\bar u)$ implies $|g|=\sqrt{2\alpha\beta}$.
The derivative of the mapping $u\mapsto  g\cdot (u-\bar u) + \frac L2(u-\bar u)^2 + \frac\alpha2 |u|^2$
at $\bar u$ is equal to $g + \alpha \bar u \ne0$.
Hence, $\bar u$ is not a global minimum of this quadratic function.
Since $\bar u\ne0$ and $u\mapsto |u|_0$ is constant near $\bar u$, it follows that
$\bar u$ is not a local minimum of \eqref{eq383}.
\end{proof}

\begin{theorem}\label{thm325}
Assume that there is no admissible control satisfying the maximum principle \eqref{eqpmp}.
Let $\bar u$ be a solution of the convexified problem \eqref{eq380}.
 Then $\bar u$ is {\em not} a fixed point of \cref{A1} and \cref{A2}.
\end{theorem}
\begin{proof}
Due to the optimality condition  $0 \in \nabla f(\bar u)(x) + \partial g_\alpha(\bar u(x))$,
it holds
\[
 \bar u(x) = \arg\min_{|u|\le b} \nabla f(\bar u)(x) \cdot u + g_\alpha(u).
\]
for almost all $x\in \Omega$.
By assumption, $\bar u$ does not satisfy the maximum principle \eqref{eqpmp}.
Hence, on a set of positive measure $A\subset \Omega$ we have that $\bar u(x)$ is not a solution of
\[
 \min_{|u|\le b}  \nabla f(\bar u)(x)\cdot u + \frac\alpha2 |u|^2 + \beta|u|_0.
\]
Then the result of \cref{lem324} implies that $\bar u$ cannot be a
local minimum of the optimization problem \eqref{eq221} for all $L\ge0$, and hence $\bar u$ is not a fixed point of the algorithms \cref{A1} and \cref{A2}.
\end{proof}

This result shows that the proximal gradient method applied to the convexified problem
might deliver suboptimal results for the original problem in the unsolvable case.
If the IHT method is started in a solution of the convexified problem that does not solve the original problem,
it will still generate points that strictly decrease the cost functional of the original problem.
Here, it is an open question how well these iterates and their weak limit points will approximate the infimum
of the cost functional.

\section{Numerical experiments}
\label{sec4}

In this section, we will present results of numerical experiments.
These were carried out in the framework of \cref{ex1}. That is, $f(u)$ is defined as
\[
 f(u) := \frac12\|y_u-y_d\|_{L^2(\Omega)}^2,
\]
where $y_u$ denotes the weak solution of the elliptic partial differential equation
\[
 -\Delta y = u \quad \text{ in } \Omega, \quad y=0 \quad \text{ on } \partial \Omega.
\]
Here, we chose $\Omega=(0,1)^2$.
The partial differential equation was discretized with piecewise linear finite elements, where the domain
was divided into a regular mesh. The controls were discretized with piecewise constant functions on the triangles.
If not mentioned otherwise, we used a discretization with $500,000$ triangles and mesh-size $h=\sqrt2/500\approx 0.0028$.
As problem data we chose
\[
 y_d(x_1,x_2) = 10 x_1 \sin(5x_1) \cos(7x_2), \quad \alpha = 0.01,\quad  \beta=0.01,\quad b = 4,
\]
which are taken from \cite{ItoKunisch2014}.
The computed solution for this problem can be seen in \cref{fig5}.
Clearly, the optimal control is discontinuous.

\begin{figure}[htb]
 \includegraphics[width=0.49\textwidth]{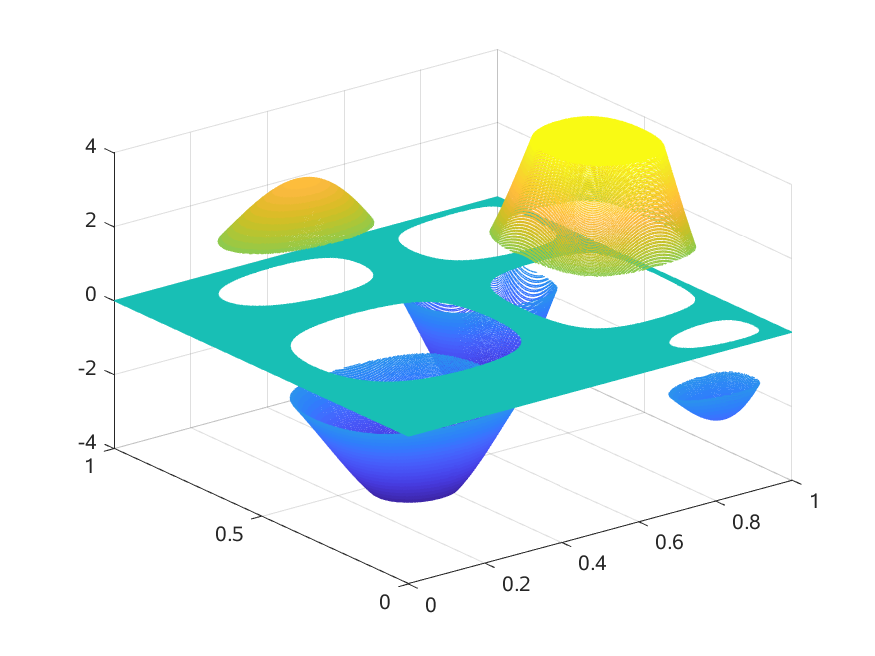}
 \includegraphics[width=0.49\textwidth]{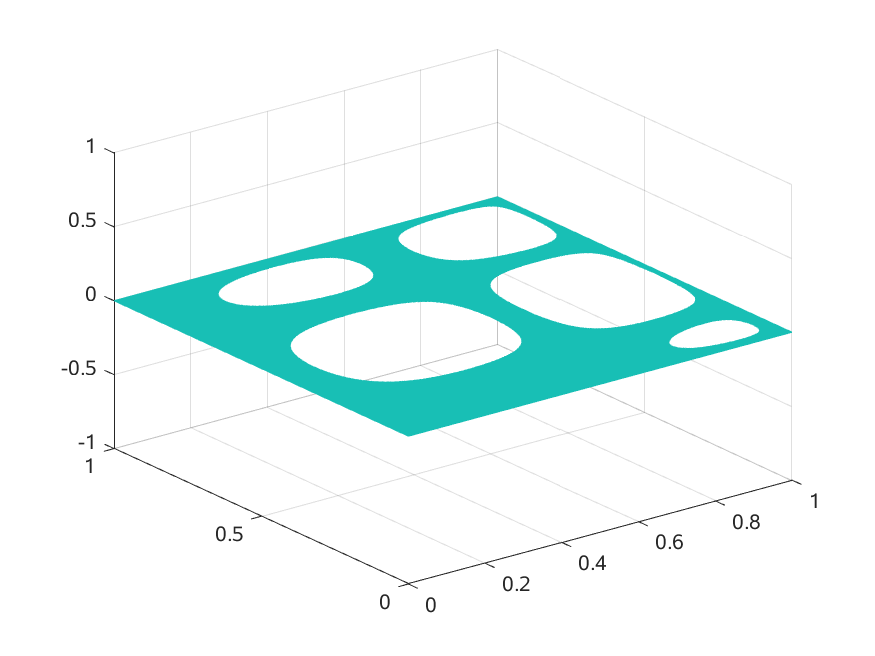}
 \caption{Optimal control $u$ (left), its support $\{x:\ u(x)\ne0\}$ (right)}
 \label{fig5}
\end{figure}

\subsection{Comparison of step-size selection strategies}
\label{sec41}

First, we will report on the different step-size selection strategies available. As already mentioned in \cref{sec33},
we tried several methods. Let us describe them in detail.
Let $\hat L^0>0$ be an initial step-size, $\theta\in(0,1)$ be a reduction factor, and $\eta>0$ a constant ruling the decrease condition.

The first strategy was simple back-tracking:
starting with $\hat L^0$, determine $L_k$ to be the largest number of the form $\hat L^0\theta^i$, $i=0,1,\dots$, that satisfies the descent
condition \eqref{eq360}. We will abbreviate this strategy by {\bf BT}.

Second, we used a widening strategy: If $\hat L^0$ satisfies \eqref{eq360}, then determine  $L_k$ to be the
largest number of the form $\hat L^0\theta^{-i}$, $i=0,1,\dots, I_{\max}$, where $I_{\max}$ is a maximal number of widening steps.
If $\hat L^0$ does not satisfy \eqref{eq360}, then compute $L_k$ according to {\bf BT}.
We will denote this strategy by {\bf BT-W}.

Third, if step-size $L=0$ satisfies  \eqref{eq360}, then set $L_k=0$ otherwise determine $L_k$ according to {\bf BT-W}.
We will denote this strategy by {\bf BT-0}.

In all our tests, we chose
\[
 \theta = 0.5, \quad \eta = 10^{-4}, \quad I_{\max}=40.
\]
In addition, \cref{A2} was stopped if
\[
 | f(x_{k+1}) + g(x_{k+1}) - (f(x_k+g(x_k)) | \le 10^{-12}.
\]
The result for computations with different line-search strategies can be found in \cref{table1}.
Here, we denoted by $u^*$ the final iterate of the method. The column 'pde' notes the
number of pde solves during the iteration. Note that the computation of the gradient of the
cost functional requires two pde solves, while one step of the line-search method requires one pde solve.
As can be seen in the table, the standard backtracking method {\bf BT} performs better for smaller initial step-size $\hat L^0$.
The linesearch with widening needs much more pde solves. This is due to the fact that for the first three iterations
$L_{\max}$ steps are done to decrease $L$. Here, the line-search strategy {\bf BT-0} that starts with $L=0$ is clearly better.
Hence, {\bf BT-0} is a good compromise to obtain small values of the objective with a small number of pde solves without tuning the initial step-size.

\begin{table}[htb]
\begin{center}
\begin{tabular}{CCRRc}
\hline
f(u^*)+g(u^*) & \|u^*\|_0 & \text{pde} & \mc{$\hat L^0$} & strategy \\
\hline
5.34002749 & 0.444590 &  31 & 0.000001 & {\bf BT} \\
5.34002749 & 0.444586 &  19 & 0.000010 & {\bf BT} \\
5.34002751 & 0.444308 &  20 & 0.000100 & {\bf BT} \\
5.34003263 & 0.439584 &  11 & 0.001000 & {\bf BT} \\
5.34047914 & 0.337838 &  31 & 0.010000 & {\bf BT} \\
5.35184756 & 0.054546 &  71 & 0.100000 & {\bf BT} \\
5.35905947 & 0.000000 &  42 & 1.000000 & {\bf BT} \\
5.35905947 & 0.000000 &  42 & 10.000000 & {\bf BT} \\
\hline
5.34002750 & 0.444602 & 154 & 0.010000 & {\bf BT-W} \\
5.34002750 & 0.444602 &  40 & 0.010000 & {\bf BT-0} \\
\hline
\end{tabular}
\caption{Comparison of line-search strategies}\label{table1}
\end{center}
\end{table}

\subsection{Comparison to \cite{ItoKunisch2014}}

Let us compare our results to computations of \cite[Example 2.14]{ItoKunisch2014},
where the influence of variations of $\beta$ were studied.
There, no control constraints are present, i.e., $b=+\infty$.
The computations were done with strategy {\bf BT-0} and $\hat L^0=0.01$.
The results can be found in \cref{table2}.
There,  $\|u^*\|_0$ denotes the result of our computations.
The column $N_0$ is taken from \cite[Example 2.14]{ItoKunisch2014}, it denotes
the number of non-zero coefficients of the control, computed with an active set-strategy for finite-difference scheme with $129\times 129$ node.
The column $N_0/129^2$ thus serves as an approximation of the $L^0$-norm of the controls computed in \cite{ItoKunisch2014}.
As can be seen in \cref{table2}, the results are in good agreement.

\begin{table}[htb]
\begin{center}
\begin{tabular}{CCRL}
\hline
\beta & \|u^*\|_0 & N_0  & \mc{$N_0/129^2$}\\
\hline
0.5000 & 0.000000 & 0 & 0.00000000\\
0.1000 & 0.068926 & 1135 & 0.06820504\\
0.0500 & 0.173892 & 2852 & 0.17138393\\
0.0100 & 0.444780 & 7296 & 0.43843519\\
0.0050 & 0.540102& 8853 & 0.53199928\\
0.0010 & 0.736796 & 12090 & 0.72651884\\
\hline
\end{tabular}
\caption{Comparison  \cite[Example 2.14]{ItoKunisch2014}}\label{table2}
\end{center}
\end{table}

\subsection{Discretization}

Next, we report on the influence of discretization on the algorithm.
We use the problem data as in \cref{sec41}.
Again, the computations were done with strategy {\bf BT-0} and $\hat L^0=0.01$.
As can be seen from \cref{table3}, the values of the objective as well as of $\|\cdot\|_0$ are converging for decreasing mesh-size $h$.
In addition, the number of pde solves until the termination criterion is reached is stable across different discretization levels.

\begin{table}
\begin{center}
\begin{tabular}{CCRRc}
\hline
h & f(u_h^*)+g(u_h^*) & \|u_h^*\|_0 & \text{pde} \\
\hline
0.1414 & 3.145593 & 0.365000 &  42  \\
0.0707 & 4.286681 & 0.428750 &  39  \\
0.0354 & 4.850340 & 0.437812 &  54  \\
0.0177 & 5.121182 & 0.443828 &  51  \\
0.0088 & 5.252535 & 0.444355 &  49  \\
0.00442 & 5.317021 & 0.444561 &  34  \\
0.00221 & 5.348944 & 0.444645 &  40  \\
0.00110 & 5.364823 & 0.444651 &  40  \\
\hline
\end{tabular}
\caption{Influence of discretization}\label{table3}
\end{center}
\end{table}

\subsection{Comparison to $L^1$-optimization problems}
\label{sec44}

In the literature, problems involving $L^1$-norms are solved to approximate $L^0$-optimization problems.
This goes back to the pioneering work \cite{Donoho2006}, where it is shown that under some condition, solutions to $L^1$-problems
solve also the $L^0$-problem.
In \cref{sec23}, we showed that both types of problems are equivalent in the case $\alpha=0$.
Here, we will compare the outcome of $L^0$-minimization with $L^1$-minimization for positive $\alpha$. That is,
we compare solutions of \eqref{eq001} to the solutions of
\be\label{eq410}
\min_{u\in U_{ad}} f(u) + \frac \alpha 2 \|u\|_{L^2(\Omega)}^2 + \beta \|u\|_{L^1(\Omega)}.
\ee
We computed solutions to the $L^1$-problem \eqref{eq410} and to the original $L^0$-problem \eqref{eq001}.
Here, we computed solutions of both problems for different values of $\beta$, i.e.,
\[
 \beta \in \{ 0.5 \cdot 0.7^l, \ l=0\dots 15\},
\]
as solutions to both problems \eqref{eq001} and \eqref{eq410} for the same value of $\beta$ are not directly comparable.
In \cref{fig4}, we plotted the pairs $( f(u), \ \|u\|_0)$ for the solutions $u$ of these problems for different
values of $\beta$.
As can be seen, the solutions of the $L^0$-problems clearly dominate those arising from the $L^1$-problems,
in the sense that for each solution $u_1$ of \eqref{eq410} to some $\beta_1$ there is a solution $u_0$ of \eqref{eq001}
to some $\beta_0$ such that
$f(u_0)\le f(u_1)$ and $\|u_0\|_0\le \|u_1\|_0$. In addition, with the exception of $u_1=0$ both inequalities are strict.

\begin{figure}[htb]
\begin{center}
 \includegraphics[width=0.6\textwidth]{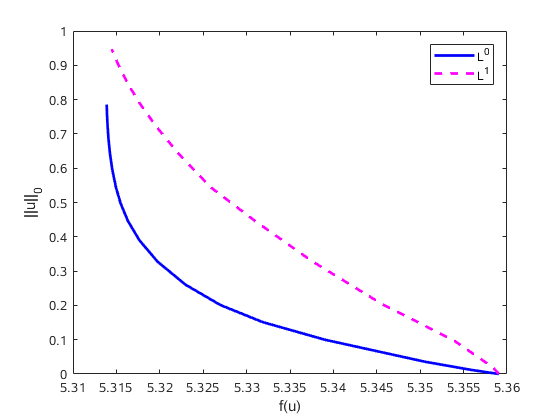}
\end{center}
\caption{Comparison of $L^0$ and $L^1$ optimization}
\label{fig4}
\end{figure}

\subsection{An unsolvable problem}

In this section, we discuss an unsolvable problem. Consider the
following optimal control problem:
Minimize the functional
\[
  \frac12\|y_u-y_d\|_{L^2(\Omega)}^2 + \frac\alpha2 \|u\|_{L^2(\Omega)}^2 + \beta\|u\|_0,
\]
where $y_u$ denotes the weak solution of the elliptic partial differential equation
with Neumann boundary conditions
\[
 -\Delta y + y = u \quad \text{ in } \Omega, \quad \frac{\partial y}{\partial n}=0 \quad \text{ on } \partial \Omega.
\]
The convexified problem as considered in \cref{sec34} is uniquely solvable, as the cost functional
$f(u) + g^{**}(u)$ is strictly convex due to the presence of the term $\|y_u-y_d\|_{L^2(\Omega)}^2$ and the injectivity of $u\mapsto y_u$.
Given $\alpha>0$, $\beta>0$, and
\[
 y_d(x) = -\sqrt{\frac \beta\alpha} - \sqrt{2 \alpha\beta},
\]
it is easy to see that the constant functions
\[
 \bar u(x) = \sqrt{\frac \beta\alpha}, \quad \bar y:=y_{\bar u}= \sqrt{\frac \beta\alpha}
\]
solve the convexified problem. The gradient $\nabla f(\bar u)$ is given by
$\nabla f(\bar u) = \bar p = - \sqrt{2 \alpha\beta}$, where the adjoint state $\bar p$ is a weak solution of
\[
 -\Delta p + p = \bar y-y_d \quad \text{ in } \Omega, \quad \frac{\partial p}{\partial n}=0 \quad \text{ on } \partial \Omega.
\]
In the light of the discussion in \cref{sec34}, the control $\bar u$ does not satisfy the maximum principle \eqref{eqpmp}.
Since $\bar u$ is the unique solution of the convexified problem it follows that the original problem is unsolvable.

We applied our \cref{A2} to this problem with $\Omega=(0,1)^2$. As predicted by \cref{thm325}, the control $\bar u$ is not
a fixed point. It turns out that for this particular example the iterates converge to the global minimizer of
\[
  \frac12\|y_u-y_d\|_{L^2(\Omega)}^2 + \frac\alpha2 \|u\|_{L^2(\Omega)}^2,
 \]
which is given by the control $\tilde u(x) =\frac1{\alpha+1}y_d(x)$.

\subsection{Application to a switching control problem}

Let us consider the following switching control problem. It was considered in \cite[Section 6]{ClasonItoKunisch2016}.
Let $\Omega=(0,1)$, $\Omega_1 = (0,1)\times(0,\frac14)$, and $\Omega_2 = (0,1)\times(\frac34,1)\times(0,1)$ be given.
The state equation is defined by
\[
 -\Delta y = \chi_{\Omega_1}(x_1,x_2) u_1(x_1) + \chi_{\Omega_2}(x_1,x_2) u_2(x_1) \quad \text{ on } \Omega,
\]
with homogeneous Dirichlet boundary conditions.
Here, two controls $u_1$ and $u_2$ are present. At each point $x_1 \in (0,1)=:I$ at most one of these two controls should be non-zero,
that is $u_1\cdot u_2=0$ should be achieved.
Following \cite{ClasonItoKunisch2016}, this switching constraint is penalized using the $L^0$-norm. The resulting optimal control problem
reads: Minimize
\[
 F(u_1,u_2):= \frac12\|y-y_d\|_{L^2(\Omega)}^2 + \frac\alpha2 \left( \|u_1\|_{L^2(I)}^2 + \|u_2\|_{L^2(I)}^2\right) + \beta \|u_1u_2\|_0,
\]
where $\|u_1u_2\|_0 = \int_I |u_1u_2|_0\dx_1$.
We will apply the proximal gradient algorithm to this problem. Here, the prox-map can be calculated pointwise again.
The scalar version of this prox map can be calculated by solving the optimization problem
\[
 \min_{u\in \R^2} g^T u + \frac L2 \|u-u_k\|_2^2 + \frac\alpha2 \|u\|_2^2 + \beta|u_1u_2|_0,
\]
which can be carried out by elementary calculations.
We used this prox-map as substitute in \cref{A2}. The step-size parameter was selected using the descent condition \eqref{eq360}.
We applied the same discretization as in the previous example. The controls $u_1$ and $u_2$ are discretized by piecewise constant
functions on a subdivision of $I$ induced by the triangulation of $\Omega$.
We took the following data \cite[Section 6]{ClasonItoKunisch2016}
\[
 y_d(x_1,x_2) = x_1 \sin(2 \pi x_1)\sin(2 \pi x_2), \quad \alpha = 10^{-5}.
\]
Using the proximal gradient algorithm with step-size strategy {\bf BT-0}, we computed solutions for different values of
$\beta$. As can be seen from \cref{fig6} and \cref{table4}, for $\beta\ge 0.1$ we got $u_1u_2=0$.

\begin{figure}[htb]
\includegraphics[width=0.32\textwidth]{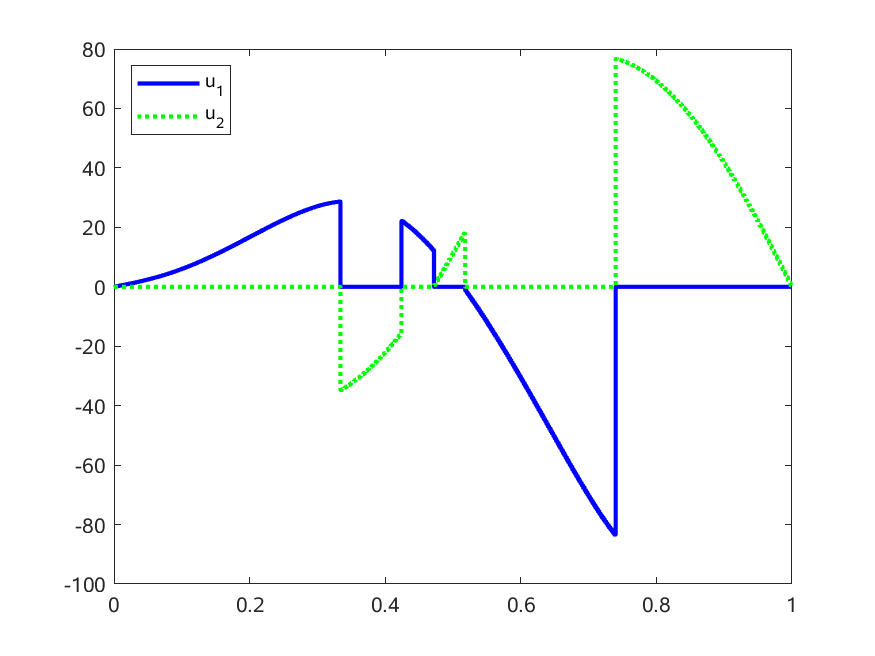}
\includegraphics[width=0.32\textwidth]{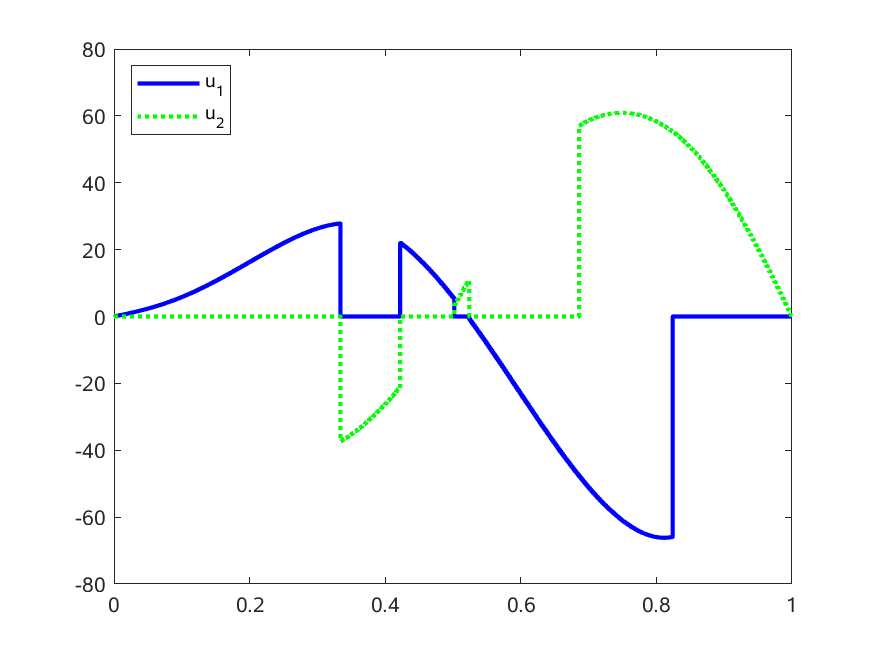}
\includegraphics[width=0.32\textwidth]{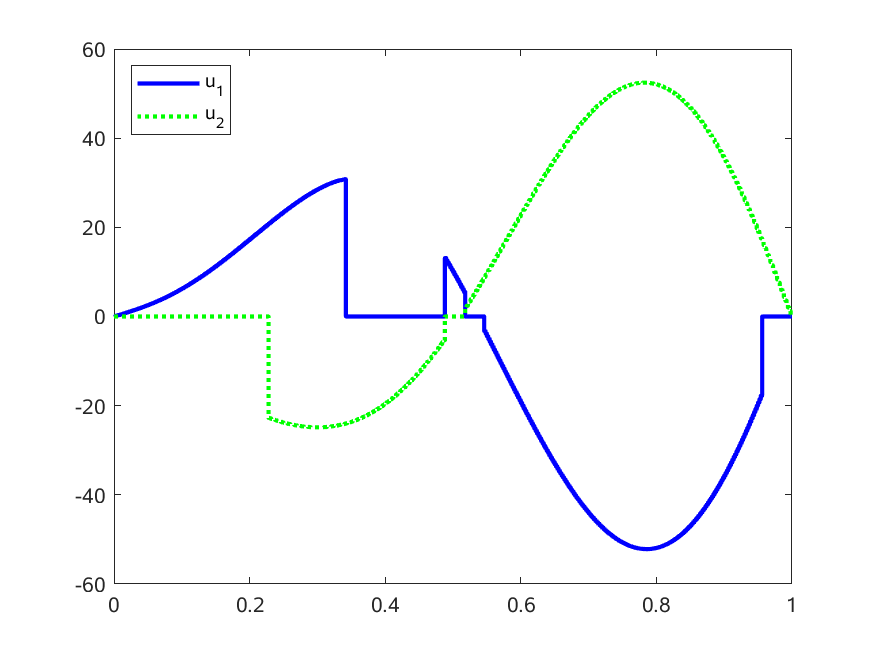}
\caption{Switching control problem: $u_1$ and $u_2$ (dotted lines)  for $\beta=10^{-1},10^{-2},10^{-3}$ (from left to right)}
\label{fig6}
\end{figure}

\begin{table}
\begin{center}
\begin{tabular}{CCC}
\hline
\beta &  F(u_1,u_2) & \|u_1u_2\|_0 \\
\hline
0.1000 & 0.024680& 0.0000 \\
0.0100 & 0.022362& 0.1380 \\
0.0010 & 0.018842& 0.5240 \\
\hline
\end{tabular}
\caption{Switching control problem: dependency on $\beta$}
\label{table4}
\end{center}
\end{table}

Unfortunately, we were not able to prove a result analogous to \cref{thm313}, as a substitute of \cref{lem27}.
Hence, the analysis of the convergence of the proximal gradient method applied to this switching control problem is subject to future research.

\bibliography{hard}
\bibliographystyle{plain_abbrv}

\end{document}